\documentclass[11pt,twoside]{amsart}
 
\usepackage{anysize} \marginsize{1.3in}{1.3in}{1in}{1in}
\usepackage{amsmath, amsthm, amssymb, hyperref, color}
\usepackage{tikz}
\usetikzlibrary{cd}
\usepackage{tikz-cd}
\usepackage{stmaryrd}
\usepackage{graphicx}
\usepackage{mathtools}
\usepackage{enumerate}
\usepackage{listings}
\usepackage{verbatim}
\usepackage{mathrsfs}

\usepackage{pgfplots}
\usepgfplotslibrary{fillbetween}

\pgfdeclarelayer{bg}
\pgfsetlayers{bg,main}

\usetikzlibrary{shapes,positioning,intersections,quotes,arrows}
\usetikzlibrary{patterns}

\usepackage[linesnumbered,lined,commentsnumbered]{algorithm2e}

\usepackage{makecell}
\usepackage{array}
\newcolumntype{?}{!{\vrule width 1pt}}
\theoremstyle{theorem}
\newtheorem{theorem}{Theorem}
\newtheorem{proposition}[theorem]{Proposition}
\newtheorem{lemma}[theorem]{Lemma}
\newtheorem{corollary}[theorem]{Corollary}
\theoremstyle{definition}
\newtheorem{definition}[theorem]{Definition}

\newtheorem{conjecture}[theorem]{Conjecture}

\newtheorem{remark}[theorem]{Remark}

\numberwithin{theorem}{section}
\newtheorem{examplex}[theorem]{Example}
\newenvironment{example}
{\pushQED{\qed}\examplex}
{\popQED\endexamplex}

\newtheoremstyle{citing}
{}
{}
{\itshape}
{}
{\bfseries}
{\textbf{.}}
{.5em}
{\thmnote{#3}}
{\theoremstyle{citing}
	\newtheorem*{custom}{}}

\newcommand{\PP}{\mathbb{P}}
\newcommand{\RR}{\mathbb{R}}

\newcommand{\CC}{\mathbb{C}}
\newcommand{\ZZ}{\mathbb{Z}}

\newcommand{\TT}{\mathbb{T}}

\DeclareMathOperator{\dml}{d_{ML}}
\DeclareMathOperator{\Trop}{Trop}

\DeclareMathOperator{\ord}{ord}

\DeclareMathOperator{\dlog}{dlog}

\DeclareMathOperator{\Spec}{Spec}
\DeclareMathOperator{\init}{in}
\DeclareMathOperator{\BS}{BS}

\DeclareMathOperator{\LCT}{LCT}

\tikzset{
  symbol/.style={
    draw=none,
    every to/.append style={
      edge node={node [sloped, allow upside down, auto=false]{$#1$}}}
  }
}

\title[MLE from a Tropical and a Bernstein--Sato Perspective]{Maximum Likelihood Estimation from a Tropical\\and a Bernstein--Sato Perspective}

\author{Anna-Laura Sattelberger}
\address{Anna-Laura Sattelberger\\Max-Planck-Institut f\"ur Mathematik in den Natur\-wis\-sen\-schaf\-ten\\Inselstra{\ss}e 22, 04103 Leipzig, Germany}
\email{anna-laura.sattelberger@mis.mpg.de}

\author{Robin van der Veer}
\address{Robin van der Veer\\KU Leuven\\Department of Mathematics\\Celestijnenlaan 200B bus 2400\\B-3001 Leuven, Belgium}
\email{robin.vanderveer@kuleuven.be}

\makeatletter
\@namedef{subjclassname@2020}{
	\textup{2020} Mathematics Subject Classification}
\makeatother

\keywords{Maximum Likelihood Estimation, Tropical Geometry, Bernstein--Sato ideal, Likelihood Geometry.\\ This is a pre-copyedited, author-produced version of an article accepted for publication in {\em International Mathmatics Research Notices} following peer review. The version of record rnac016 is available online at: \href{https://doi.org/10.1093/imrn/rnac016}{https://doi.org/10.1093/imrn/rnac016}}

\begin{document}
	\maketitle

\begin{abstract}In this article, we investigate Maximum Likelihood Estimation with tools from Tropical Geometry and Bernstein--Sato Theory.  We investigate the critical points of very affine varieties and study their asymptotic behavior. We relate these asymptotics to particular rays in the tropical variety as well as to Bernstein--Sato ideals and give a connection to Maximum Likelihood Estimation in Statistics. 
\end{abstract}

\vspace*{-1mm}
\setcounter{tocdepth}{1}
\tableofcontents

\vspace*{-7mm}

\section{Introduction}
\subsection*{Maximum likelihood estimation}
Let $X$ be a smooth closed subvariety of the algebraic torus $(\CC^*)^p$ and denote by $t_1,\dots,t_p$ the coordinates on the torus restricted to $X$. Given $\alpha\in\CC^p$, the {\em maximum likelihood estimation (MLE)}  problem is to find the zeroes of the logarithmic one-form $\dlog(t_1^{\alpha_1}\cdots t_p^{\alpha_p})$. These zeroes are also referred to as {\em critical points}. As the terminology suggests, this problem arises from Statistics. We refer to~\cite{HS14} for further details and examples related to this statistical context. In this article, we investigate this problem from an algebro-geometric point of view.
The geometric object underlying our analysis is the {\em critical locus} $C\subseteq X\times\PP^{p-1}$, consisting of all pairs $(x,\alpha)$ for which the logarithmic differential $\dlog t^\alpha$ vanishes at~$x$. As proven by Franecki--Kapranov in \cite{FK} and later by Huh in~\cite{huh_2013}, for a general data vector~$\alpha$, the number of critical points is finite and equal to the {\em signed Euler characteristic} $(-1)^{\dim X}\chi(X)$ of $X$.  We refer to non-general data vectors as {\em special} values.
The number of critical points for general $\alpha$ is also called the {\em maximum likelihood degree} of~$X$,  denoted by $\dml(X)$.  
\subsection*{Maximum likelihood estimation approaching non-general data vectors}
In this article, we study the behavior of the critical points when approaching a special data vector along a curve consisting of general values.
More precisely, we study the subset $S_F\subseteq \PP^{p-1}$ consisting of all data vectors $\alpha$ for which at least one of the $\dml(X)$ many critical points leaves $X$ as we approach~$\alpha$. We refer to the components of $S_F$ that are of codimension one as {\em critical slopes} of~$F.$ In this case, at least one of the coordinates of some critical point approaches either zero or infinity. To obtain more refined information of the asymptotic behavior, we keep track of the direction in which this critical point escapes by recording it in a set $Q_{F,\alpha}\subseteq \ZZ^p$. The points in $Q_{F,\alpha}$ essentially describe  which torus orbits in a tropical compactification of~$X$ are approached by the critical points as one approaches~$\alpha$. In Section~\ref{section asymptotic}, we introduce both of these objects in the general setting of a smooth variety $X$ with a tuple of nowhere vanishing regular functions $F=(f_1,\dots,f_p)$ on it. It turns out that both $S_F$ and $Q_{F,\alpha}$ are closely related to the divisorial valuations corresponding to the irreducible components of the boundary in a smooth compactification of $X$ with simple normal crossings (SNC) boundary. The critical points for special data vectors were, among others, also studied in~\cite{cohen_denham_falk_varchenko_2011} for the case of hyperplane arrangements and in \cite{RT18} using probabilistic methods.
To study $S_F$ and $Q_{F,\alpha}$, a good understanding of a smooth compactification and its boundary components is essential, which we tackle in Section~\ref{section MLEveryaffine}. 

\subsection*{Sch\"on varieties and tropical compactifications}
A natural class of compactifications of closed subvarieties $X$ of tori is provided by the tropical compactifications of~\cite{Tev}. Tropical compactifications are constructed by taking the closure $X^\Sigma$ of $X$ in the toric variety associated to a fan $\Sigma$ whose support is the tropical variety of~$X$. In general, these compactifications are not smooth. If the variety $X$ is {\em sch\"on} (see \cite[Definition 3.6]{huh_2013}), it allows for a  tropical compactification that is smooth and whose boundary is a simple normal crossings divisor. This class of subvarieties of tori provides a common generalization of hyperplane arrangement complements and of Newton non-degenerate hypersurfaces as shown in~\cite{Hov77}. For these compactifications, we prove the following theorem.
\begin{custom}[Theorem~\ref{mainthm}]
Let $E_i$ be an irreducible component of the boundary in a smooth tropical compactification of $X$ with simple normal crossings boundary.
Assume that $\chi(E_i^\circ)\not=0$, where $E_i^\circ \coloneqq E_i \setminus \cup_{j\neq i} E_j$. Then for general $\alpha$ in the hyperplane $H_{E_i}=\{\ord_{E_i}(t_1)s_1+\dots + \ord_{E_i}(t_p)s_p=0\}$, the vector $(\ord_{E_i}(t_1),\dots, \ord_{E_i}(t_p))$ is contained in $Q_{F,\alpha}$. In particular, the hyperplane $H_{E_i}$ is contained in $S_F$. Moreover, these hyperplanes are the only codimension-one components of~$S_F$.
\end{custom}
The toric variety associated to $\Sigma$ has codimension-one torus orbits $\mathcal{O}_\tau$ indexed by the rays $\tau\in \Sigma$. The boundary component $E_i^\circ$ is an irreducible component of $X^\Sigma\cap\mathcal{O}_\tau$ for some $\tau\in\Sigma$ and thus the condition that $\chi(E_i^\circ)\not=0$ is a property of the ray $\tau$ in the tropical variety of~$X$. To translate this condition into a tropical condition, we define the notion of \textit{rigid} rays in $\Trop(X)$;  these are the rays for which any small perturbation of the ray changes the associated initial ideal.
\begin{custom}[Proposition~\ref{main}]
	Let $X$ be a sch\"{o}n very affine variety and $\Sigma$ a fan supported on the tropical variety of $X$ such that the closure $X^{\Sigma}$ of $X$ in the toric variety associated to $\Sigma$ is smooth and $X\setminus X^\Sigma$ is a SNC divisor.
Assume that for all $\tau\in \Sigma$, the intersection $X^{\Sigma}\cap\mathcal{O}_\tau$ is connected. If $\tau$ is a rigid ray, then for general $\alpha$ in the hyperplane orthogonal to $\tau$, the primitive generator of the ray lies in $Q_{F,\alpha}$. In particular, $\PP(\tau^\perp)\subseteq S_F$. This gives a bijection of the rigid rays in $\Trop(X)$ and the codimension-one components of~$S_F$.
\end{custom}
This proposition completely determines the critical slopes of $F$ in terms of tropical data associated to $X$. By the very definition of $Q_{F,\alpha}$, this proposition gives a description of the asymptotic behavior of the critical points as we approach the special set of data vectors formed by the hyperplane orthogonal to a rigid ray. We refer to Section~\ref{subsection zeroeslog} for a more explicit interpretation of this result in terms of maximum likelihood estimation.

\subsection*{Slopes of Bernstein--Sato varieties}
In Section~\ref{section BernsteinSato}, we relate these results to\linebreak\mbox{Bernstein--Sato} ideals. For a tuple $F=(f_1,\dots,f_p)$ of regular functions on a smooth algebraic variety, the {\em Bernstein--Sato ideal}  of $F$ is the ideal $B_F$ in $\CC[s_1,\dots,s_p]$ consisting of polynomials~$b$ for which there exists a global algebraic linear partial differential operator $P(s_1,\dots,s_p)$ such that
$$P\bullet \left( f_1^{s_1+1}\cdots f_p^{s_p+1}\right) = b \cdot f_1^{s_1}\cdots f_{p}^{s_p}.$$
This definition is a generalization of the Bernstein--Sato polynomial of a single regular function $f$. 
It is an intricate invariant of the tuple $F$ which is related to the singularities of the hypersurface $V(f_1\cdots f_p)$, see for instance~\cite{BudurBSLS} or \cite{bvwz} for relations to monodromy eigenvalues.  Sabbah~\cite{Sabbah} showed that the Bernstein--Sato ideal is non-zero and that the codimension-one components of the Bernstein--Sato variety $V(B_F)$ are affine hyperplanes with rational coefficients. The set of {\em Bernstein--Sato slopes of $F$}, denoted by $\BS_F\subseteq \CC^p$, is the union of these hyperplanes after translating them to the origin. We denote by $\PP(\BS_F)$ the projectivization of this set. Maisonobe~\cite{Mai16} gave a geometric description of the Bernstein--Sato slopes. We denote by $Y$ the closure of $X$ inside $\CC^p$ and will assume that $Y$ is smooth. We then study the Bernstein--Sato ideal of the tuple of coordinate functions on $\CC^p$ restricted to $Y$. In this setup, we deduce the following theorem using Maisonobe's description of the Bernstein--Sato slopes.
\begin{custom}[Theorem~\ref{bs-sf}]
Under the assumptions of Proposition~\ref{main}, the irreducible components of $S_F\cap \mathbb{P}(\BS_F)$ are exactly the hyperplanes $\mathbb{P}(\tau^\perp)$ for $\tau$ a rigid ray contained in~$\RR_{\geq 0}^p$.
\end{custom}
A study of $\BS_F$ using the critical locus was also undertaken in \cite{BMM} under a different technical assumption, namely that the tuple $F$ is \textit{sans \'{e}clatements en codimension $0$}. We would like to point out that Lemma 2.6 therein, treating the special case $p=2$, is analogous to our Theorem~\ref{bs-sf}. In Example~\ref{example incomparable}, we demonstrate that, in general, the sets $S_F$ and $\mathbb{P}(\BS_F)$ are incomparable in the sense that either can contain irreducible components not contained in the other one. Theorem~\ref{bs-sf} explains in a rigorous way the observations made in~\cite[Example~3.1]{SatStu19}. We revisit this example in Section~\ref{section examples}.

Theorem~\ref{bs-sf} provides information on the slopes of the Bernstein--Sato variety, but not on the affine translation with which these slopes appear. In Subsection~\ref{subsection LCT}, we formulate a conjecture related to these affine translates. This conjecture is formulated in terms of the \textit{log-canonical threshold polytope} of the tuple~$F$. We prove this conjecture for the case of indecomposable central hyperplane arrangements, in which case it proves  \cite[Conjecture 3]{BudurBSLS} for complete factorizations of hyperplane arrangements.
It turns out that in this case the object involved has already extensively been studied under a different name: it is the {\em matroid polytope} of \cite{Feichtner2005}, and Proposition 2.4 in loc. cit. precisely recovers our conjecture. The Bernstein--Sato ideal of hyperplane arrangements has recently also been studied using different methods in~\cite{Bath19}, \cite{maisHPA},  and~\cite{Wu20}.

\bigskip
In summary, our results connect Tropical Geometry, Bernstein--Sato Theory, and Likelihood Geometry. Among others, our article provides new tools for Algebraic Statistics and Particle Physics: in the recent article~\cite{ST20}, Sturmfels and Telen outline a link of MLE for discrete statistical models to scattering amplitudes. 

\subsection*{Notation and conventions}
By a {\em variety} we mean an integral, separated scheme of finite type over the complex numbers, unless explicitly stated otherwise. A property of a variety is called {\em general} if it holds true on a Zariski dense open subset of the variety. For a smooth variety $X$ with compactification $X\hookrightarrow Y$ s.t. $Y$ is smooth and $E\coloneqq Y \setminus X =   \bigcup_{i=1}^q E_i$ is a divisor with irreducible components $E_i$, we denote by
$$E_i^\circ \coloneqq E_i\setminus \bigcup_{j\not=i}E_j$$
the complement of $E_i$ by the other components. When $X$ is equipped with a tuple  $F=(f_1,
\dots,f_p)$ of regular functions, we denote by $H_{E_i}$ the hyperplane
$$H_{E_i}  \coloneqq  \left\{\ord_{E_i}(f_1)s_1+\dots+\ord_{E_i}(f_p)s_p=0\right\} \subseteq  \PP^{p-1}.$$
We denote by $\Delta$ the formal disc $\Spec \CC\llbracket t\rrbracket $, by $\Delta^\circ\coloneqq \Spec \CC (\!( t)\!)$ its generic point, and by $0$ its closed point.

Throughout the article, we will always assume that all cones in a fan are strongly convex.

\section{Asymptotic behavior of critical points}\label{section asymptotic}
In this section, we introduce the objects of study and investigate basic properties.
We study smooth varieties $X$ with a $p$-tuple $F=(f_1,\dots,f_p)$ of nowhere vanishing regular functions on $X$. 
For $\alpha = (\alpha_1,\ldots,\alpha_p)\in \CC^p$, 
$$\dlog f^\alpha  =   \sum_{i=1}^p\alpha_i\frac{df_i}{f_i} \in H^0\left( X,\Omega_X^1\right)$$
denotes the {\em logarithmic differential} of $f^{\alpha}\coloneqq f_1^{\alpha_1}\cdots f_p^{\alpha_p}$.
\begin{definition}\label{def critical locus}
The {\em critical locus} of $F$, encoding the zeros of $\dlog f^\alpha$, is defined to be
$$C_{F}  \coloneqq  \left\{(x,\alpha)\in X\times\PP^{p-1}\mid \dlog f^\alpha (x)=0  \right\} \subseteq X\times \PP^{p-1}.$$
\end{definition}
Let $X\hookrightarrow Y$ be a compactification of $X$ and denote by $E\coloneqq Y\setminus X$ the boundary of this compactification.
\begin{definition}\label{def asymptotic critical locus}
The {\em asymptotic critical locus} of $F$ with respect to $Y$, denoted $C_{F,Y}$, is the closure of $C_F$ inside $Y\times\PP^{p-1}$, i.e., $C_{F,Y}\coloneqq \overline{C_F}^{Y\times \mathbb{P}^{p-1}}$.
\end{definition}
Denote by $\pi_1:Y\times\PP^{p-1}\to Y$ (resp. $\pi_2:Y\times\PP^{p-1}\to \PP^{p-1}$) the projection to the first (resp. second) factor.
\begin{definition}\label{slopes}
Associated to $F$ we define the variety
$$S_F \coloneqq \pi_2 \left( C_{F,Y}\cap \pi_1^{-1}(E)\right)   \subseteq  \mathbb{P}^{p-1}$$
and refer to its irreducible  components of codimension one as  {\em critical slopes} of $F.$
\end{definition}
The variety $S_F$ is the image of a proper variety and as such a closed subvariety of~$\mathbb{P}^{p-1}$.

	\begin{example}\label{ex:hpa}
		Let $X'\subseteq \CC^2$ be the complement of the arrangement defined by\linebreak$f=xy(x-y)(x-1)$, the four factors of which form the tuple $F=(f_1,f_2,f_3,f_4)$. As compactification, we take $\mathbb{P}^2$. With a computer algebra system one confirms that the critical slopes of $F$ are $$\left\{s_1+s_2+s_3+s_4=0\right\} \cup  \left\{s_2+s_3=0\right\} \cup \left\{s_1+s_2+s_3=0\right\} \cup \bigcup_{i=2}^4\left\{s_i=0\right\}.$$
More precisely, $\{s_1+s_2+s_3+s_4=0\} = \pi_2(C_{F,\mathbb{P}^2}\cap \pi_1^{-1}(H_\infty))$, $\{s_2+s_3=0\}=\pi_2(C_{F,\mathbb{P}^2}\cap \pi_1^{-1}([0:1:0])),$ $\{s_1+s_2+s_3=0\}= \pi_2(C_{F,\mathbb{P}^2}\cap \pi_1^{-1}([0:0:1])),$ and $\{s_i=0\} $ comes from $\{f_i=0\}$ for $i=2,3,4.$
		\end{example}
 
\begin{lemma}\label{closureLemma}
Let $Z$ be a variety, $X\subseteq  Z$ a locally closed subvariety and  $x\in Z$ a closed point. Then $x \in \overline{X}$ if and only if there exists a morphism $\gamma:\Delta \to Z$ such that $\gamma(\Delta^\circ)\in X$ and $\gamma(0)=x$.
\end{lemma}
\begin{proof}
The existence of such $\gamma$ clearly implies that $x\in \overline{X}$. 
For the reverse implication, let $x\in \overline{X}$. Take a curve $C\subseteq Z$ with $x\in C$ and $C\cap X\not=\emptyset$. Let $\tilde{C}$ denote the normalization of $C$ and choose a point $\tilde{x}\in\tilde{C}$ lying over~$x$. Since $\tilde{C}$ is smooth, by the Cohen structure theorem we can construct the required morphism as
$\mathcal{O}_{\overline{X},x}\to\widehat{\mathcal{O}}_{\overline{X},x}\to  \widehat{\mathcal{O}}_{C,x}\to  \widehat{\mathcal{O}}_{\tilde{C},\tilde{x}} \cong  \CC\llbracket t \rrbracket .$
\end{proof}
As an immediate consequence, we deduce the following corollary.
\begin{corollary}\label{jets}
A point $\alpha\in\PP^{p-1}$ lies in $S_F$ if and only if there exists a morphism \mbox{$\gamma\colon \Delta\to Y\times \PP^{p-1}$} such that $\gamma(\Delta^\circ)\in  C_F$ and $\gamma(0)\in E\times \{\alpha\}$. 
\end{corollary}
To keep track of more refined information about the asymptotic behavior of the critical points, we introduce the following subsets of $\ZZ^p$.
\begin{definition}\label{QF}
Let $Y$ be a compactification of $X$ and let $\alpha\in\mathbb{P}^{p-1}$. Let $\pi_1\colon Y\times\PP^{p-1}\to Y$ denote the projection to the first factor. We denote by $Q_{F,\alpha}\subseteq\mathbb{Z}^p$ the set of vectors that arise in the following manner. A vector is in $Q_{F,\alpha}$ if and only if it is of the form $ (\ord_t(\gamma^*(\pi_1^*f_1)),\dots, \ord_t(\gamma^*(\pi_1^*f_p)))$  for some morphism $\gamma:\Delta \to Y\times\PP^{p-1}$ such that $\gamma(\Delta^\circ)\in  C_F$ and $\gamma(0)\in E\times \{\alpha\}$.
\end{definition}
For brevity, we denote $ (\ord_t(\gamma^*(\pi_1^*f_1)),\dots, \ord_t(\gamma^*(\pi_1^*f_p))$ by $\ord_t(\gamma^*(\pi_1^*F))$. In general, it is difficult to compute $Q_{F,\alpha}$ explicitly. In Theorem~\ref{mainthm} and Proposition~\ref{main}, we will recover it partially.
The following lemma justifies suppressing the compactification in the notations $S_F$ and~$Q_{F,\alpha}$.
\begin{lemma}
\label{independence}
The sets $S_F$ and $Q_{F,\alpha}$, $\alpha\in\PP^{p-1}$, do not depend on the choice of a compactification.
\end{lemma}
\begin{proof} 
Let $j_1:X\hookrightarrow Y_1$ and $j_2:X\hookrightarrow Y_2$ be two compactifications of $X$ with boundaries $E_i=Y_i\setminus X$ and projections $\pi_{i,2}:C_{F,Y_i}\to \mathbb{P}^{p-1}$,  $i=1,2$. Let $\alpha\in \pi_{1,2}(C_{F,Y_1}\cap \pi_{1,1}^{-1}(E_1))$. 
By Corollary~\ref{jets} this is equivalent to the existence of a morphism
$\gamma \colon\Delta  \to Y_1 \times \mathbb{P}^{p-1},$
such that
\begin{center}
(1) $\gamma(\Delta^\circ)\in C_F\subseteq C_{F,Y_1}$ \quad (2) $\pi_{1,1}(\gamma(0))\in E_1$, and \quad (3) $\pi_{1,2}(\gamma(0))=\alpha.$
\end{center}
By properness of $Y_2\times\PP^{p-1}$, the restricted morphism $\gamma^{\circ}:\Delta^{\circ}\to C_{F}$ can be extended to a morphism $\tilde\gamma:\Delta \to Y_2\times\PP^{p-1}$. We need to show that $\tilde{\gamma}(0)\in E_2\times\{\alpha\}$. Assume that $\pi_{2,1}(\tilde\gamma(0))\not\in E_2$. Then $\pi_{2,1}(\tilde\gamma(0))\in X$. But this is not possible, since then also $\pi_{1,1}(\gamma(0))\in X$, since $\tilde{\gamma}$ and $\gamma$ coincide on $\Delta^\circ$. Hence $\pi_{2,1}(\tilde\gamma(0))\in E_2$. Similarly, we conclude that $\pi_{2,2}(\tilde{\gamma}(0))=\alpha$, so that $\alpha\in  \pi_{2,2}(C_{F,Y_2}\cap \pi_{2,1}^{-1}(E_2))$. 
The statement for $Q_{F,\alpha}$ follows from the same argument, since $\ord_t(\gamma^*(\pi_{1,1}^*(f_i))=\ord_t(\gamma^*(\pi_{2,1}^*(f_i))$ for all $i=1,\dots,p$.
\end{proof}
The following observation is immediate from the definitions.
\begin{lemma}\label{SFasQF}
The closed points of $S_F$ are given by those $\alpha \in \PP^{p-1}$ for which $Q_{F,\alpha}$ is non-empty, i.e.,
$S_F = \{\alpha\in\PP^{p-1}\mid Q_{F,\alpha}\not=\emptyset\}.$
\end{lemma}
We recall that for a boundary component $E_j$ in a compactification of $X,$ we denote by $H_{E_j}$ the hyperplane $\left\{\ord_{E_j}(f_1)s_1+\dots+\ord_{E_j}(f_p)s_p=0\right\} $ in $\PP^{p-1}.$
\begin{proposition}\label{SFBound} 
Suppose that $j:X\hookrightarrow Y$ is a compactification with $Y$ smooth and $Y\setminus X$ a simple normal crossings divisor $E$ with irreducible components $E_1,\dots,E_q$. 
Suppose that $\gamma:\Delta \to Y\times\PP^{p-1}$ is such that $\gamma(\Delta^\circ)\in  C_F$ and 
$$\gamma(0) \in \bigcap_{j\in J}E_i\times \{\alpha\}$$
with $J\subseteq \{1,\dots,q\}$.
 Then $\alpha$ is contained in all hyperplanes $\{H_{E_j}\}_{j\in J}$, i.e.,\linebreak
\mbox{$\alpha\in \bigcap_{j\in J}H_{E_j}.$}
In particular, $S_F \subseteq \bigcup_{j=1}^q H_{E_j}$ and hence also the critical slopes of $F$ are among the $H_{E_j}$.
\end{proposition}
\begin{proof}
Since $E$ is a SNC divisor, the locally free sheaf $\Omega_X^1$ on $X$ extends to the locally free sheaf $\Omega_Y^1(\log E)\subseteq j_*\Omega_X^1$ on $Y$, for which $\Omega_Y^1(\log(E))|_X= \Omega_X^1$. It can be verified in local coordinates that for all $\alpha\in\CC^p$, the form $\dlog f^\alpha\in H^0(Y, j_*\Omega_X^1)$ lies in $H^0(Y,\Omega_Y^1(\log E))$.
We then have the following commutative diagram:
\[
\begin{tikzcd}
\CC^p\arrow[r]\arrow[dr] &H^0\left(Y,\Omega_Y^1\left(\log E\right)\right)\arrow[d]\\
& H^0\left(X,\Omega_Y^1\left(\log E\right)\right)=H^0\left(X,\Omega_X^1\right)
\end{tikzcd}
\]
The horizontal arrow defines the incidence variety
$$C_{F,Y}(\log E)  \coloneqq \left\{(y,\alpha)\mid \dlog f^\alpha(y)=0\right\}\subseteq Y\times \mathbb{P}^{p-1},$$
where $ \dlog f^\alpha$ is regarded as a global section of $\Omega_Y^1(\log E)$, and hence $\dlog f^\alpha(y)$ is an element of the fiber of this locally free sheaf over $y$. We refer to $C_{F,Y}(\log E)$ as the {\em logarithmic critical locus of $F$}.
This is a closed, possibly reducible, subvariety of $Y\times\mathbb{P}^{p-1}$. By the commutativity of the diagram above, $C_F\subseteq C_{F,Y}(\log E)$ and hence also $C_F \subseteq C_{F,Y} \subseteq C_{F,Y}(\log E)$.

Let $y\in \bigcap_{j\in J} E_j$. Assume for simplicity of notation that $J=\{1,\dots,r\}$. Take local coordinates $y_1,\dots,y_n$ on a small open $U$ around $y$ in which $E_i\cap U=\{y_i=0\}$ for $i=1,\dots,r$. Then in these coordinates
$$\frac{df_i}{f_i} = \frac{d(u_i y_1^{\ord_{E_1}(f_i)}\cdots y_r^{\ord_{E_r}(f_i)})}{u_i y_1^{\ord_{E_1}(f_i)}\cdots y_r^{\ord_{E_r}(f_i)}}=\frac{du_i}{u_i}+\ord_{E_1}(f_i)\frac{dy_1}{y_1}+\dots +\ord_{E_r}(f_i)\frac{dy_r}{y_r},$$
where the $u_i$ are non-vanishing functions on $U$. Summing up, we  conclude that
\begin{align}\label{dloglocal}
\dlog f^\alpha = \theta(\alpha)+\sum_{j=1}^r \frac{(\ord_{E_j}(f_1)\alpha_1+\dots+\ord_{E_j}(f_p)\alpha_p)dy_j}{y_j},
\end{align}
where $\theta(\alpha)$ is a regular $1$-form around $y$, which we write in coordinates as $\sum_{i=1}^n \theta_i(\alpha)dy_i$. We will denote $q_j(\alpha)=\ord_{E_j}(f_1)\alpha_1+\dots+\ord_{E_j}(f_p)\alpha_p$. In the frame $dy_1/y_1,\ldots,dy_r/y_r,$ $dy_{r+1},\dots,dy_n$ for $\Omega^1_Y(\log E)$ we conclude that $\dlog f^\alpha$ is given by \linebreak$y_1\theta_1(\alpha)+q_1,\dots,y_r\theta_r(\alpha)+q_r,\theta_{r+1}(\alpha),\dots,\theta_n(\alpha)$, and thus these are a system of defining equations for $C_{F,Y}(\log E)\cap \pi_1^{-1}(U)$. It follows that 
\begin{align}\begin{split}
C_{F,Y}(\log E)\cap \pi_1^{-1}(\cap_{j\in J}E_j\cap U)&=V(y_1,\dots,y_r,q_1,\dots,q_r,\theta_{r+1},\dots,\theta_n)\label{CFYLBound}\\
& \subseteq \left(\cap_{j\in J}E_j\cap U\right)\times \left(\cap_{j\in J}H_{E_j}\right).
\end{split}
\end{align}
Since $y\in  \bigcap_{j\in J} E_j$ was arbitrary, we conclude that 
$$C_{F,Y}\left(\log E\right)\cap \pi_1^{-1}\left(\cap_{j\in J}E_j\right) \subseteq \left(\cap_{j\in J}E_j\right)\times \left(\cap_{j\in J}H_{E_j}\right)$$
and thus also
$$C_{F,Y}\cap \pi_1^{-1}\left(\cap_{j\in J}E_j\right) \subseteq \left(\cap_{j\in J}E_j\right)\times \left(\cap_{j\in J}H_{E_j}\right).$$
For a morphism $\gamma:\Delta\to Y\times\PP^{p-1}$ as in the statement of the proposition, we thus conclude that
\begin{align*}
\gamma(0) \in C_{F,Y}\cap \left(\cap_{j\in J}E_j\times\{\alpha\}\right) &=C_{F,Y}\cap \pi_1^{-1}\left(\cap_{j\in J}E_j\right)\cap \pi_2^{-1}(\alpha)\\
&\subseteq \pi_1^{-1}\left(\cap_{j\in J}E_j\right)\times \left(\cap_{j\in J}H_{E_j}\cap \{\alpha\}\right).
\end{align*}
In particular, $(\cap_{j\in J}H_{E_j})\cap \{\alpha\}$ is non-empty so that indeed $\alpha\in \cap_{j\in J}H_{E_j}$.
The last claim of the lemma follows from the first statement and Lemma~\ref{jets}. 
\end{proof}
\begin{remark}
In the above proof, the reason we work with $C_{F,Y}(\log E)$ is simply to give a bound for $C_{F,Y}$. The upshot of $C_{F,Y}(\log E)$ is that $C_{F,Y}(\log E)\cap \pi_1^{-1}(\cap_{j\in J}E_j\cap U)$ is easy to compute, as in Equation~\eqref{CFYLBound}. On the other hand, computing $C_{F,Y}\cap \pi_1^{-1}(\cap_{j\in J}E_j\cap U)$, which is our actual goal, is much more difficult. The reason for this difference is essentially that $C_{F,Y}$ is equal to $\overline{C_{F,Y}(\log E)\setminus \pi_1^{-1}(E)}$, and hence is defined by a saturation of the ideal defining $C_{F,Y}(\log E)$. The ideal defining $C_{F,Y}(\log E)$ can be written down fairly explicitly as in the proof above, but we do not have a general method to analyze the saturation.

We illustrate with an example that in general there is a difference between $C_{F,Y}$ and $C_{F,Y}(\log E)$. Consider the variety $U=\CC^2\setminus V(f_1f_2)$, with $f_1=x$ and $ f_2=xy-1.$ $U$~can be compactified in $\PP^2$, with boundary components $V(x),$  $V(xy-z^2),$ and $V(z)$. We can perform blowups with exceptional locus lying over $V(z)$ to make the boundary~$E$ SNC. On can compute that 
$$C_{F,Y}\cap \pi_1^{-1}(\CC^2)=\{y=0\}\times \{s_2=0\},$$
while
$$C_{F,Y}(\log E)\cap \pi_1^{-1}(\CC^2)=\{y=0\}\times \{s_2=0\}\cup \{x=0\}\times\{s_2=0\}.$$
\end{remark}

\section{Maximum likelihood estimation on sch\"on very affine varieties}\label{section MLEveryaffine}
In this section, we investigate {\em sch\"on} very affine varieties. Those varieties allow for a smooth compactification with SNC boundary obtained from  combinatorial data, namely from their tropical variety. For background in Tropical Geometry and tropical compactifications, we refer the reader to \cite{MS15}, \cite{Tev}, and~\cite{OnTropComp}. In particular, we refer to \cite[Definitions 3.1.1 and 3.2.1]{MS15} for the definition of the tropical variety and to \cite[Theorem 3.2.3]{MS15} for equivalent characterizations.
We analyze the MLE problem in terms of Tropical Geometry. The main results of this section are Theorem~\ref{mainthm} and Proposition~\ref{main} which completely determine the critical slopes from tropical~data. 

\subsection{Zeroes of logarithmic differential forms}\label{subsection zeroeslog}
Let $X\subseteq (\CC^*)^p$ be a smooth closed subvariety of the algebraic $p$-torus. We denote by $t_1,\dots,t_p$ the coordinate functions on $(\CC^*)^p$ and by $f_i\coloneqq t_i|_X$ their restrictions to~$X$. We will assume that there exists a fan structure $\Sigma$ on $\Trop(X)$ such that the closure $X^{\Sigma}$ of $X$ in the associated toric variety $ \TT^\Sigma $ is proper and smooth and $X^{\Sigma}\setminus X=X^\Sigma\cap (\TT^\Sigma\setminus (\CC^*)^p)$ is a reduced simple normal crossings divisor. If $X$ is sch\"on, it admits such a compactification  (see proof of \cite[Theorem 2.5]{HomTropVar}). We denote by~$\{\tau_i\}_{i\in I}$ the rays in $\Sigma$ and we denote the primitive ray generator of the ray $\tau$ by $v_\tau\in \ZZ^p$.  

The irreducible components of the boundary $E=X^\Sigma\setminus X$ are partitioned by the rays in~$\Sigma$. Those  corresponding to $\tau$ are the irreducible components $\{E_{\tau,i}\}$ of $E_\tau \coloneqq X^{\Sigma}\cap \overline{\mathcal{O}_\tau}$, where $\mathcal{O}_\tau \subseteq \mathbb{T}^\Sigma$ is the locally closed torus orbit corresponding to $\tau$. Recall that an alternative characterization of sch\"on very affine varieties is as those closed subvarieties of a torus for which there is a fan structure $\Sigma$ on $\Trop(X)$ such that the multiplication map $m:X^\Sigma\times (\CC^*)^p\to \TT^\Sigma$ is smooth. In particular, it follows from this characterization that every $E_\tau$ is smooth, and hence if $E_\tau$ is not irreducible, it is a disjoint union of smooth irreducible components. 

By construction of $\TT^\Sigma$ we have that $\ord_{\overline{\mathcal{O}_{\tau}}}(t_j)=v_{\tau_j}$. Since $E_{\tau}$ is reduced by assumption, it follows that also for every component $E_{\tau,i}\subseteq E_\tau$, $\ord_{E_{\tau,i}}(f_j)=(v_\tau)_j$. In particular, for each such $E_{\tau,i}$, $H_{E_{\tau,i}}$ equals $\PP(\tau^\perp)$, the hyperplane orthogonal to $\tau$. 

Recall that for each $\alpha\in \CC^p$, the form $\dlog f^\alpha$ extends to a global section of $\Omega_{X^{\Sigma}}^1(\log E)$, i.e., a differential one-form on $X^{\Sigma}$ with logarithmic poles along $E$. As in the proof of Proposition~\ref{SFBound}, we denote by~$C_{F,X^{\Sigma}}(\log E) $ the  logarithmic critical locus of $F,$ i.e.,
$$C_{F,X^{\Sigma}}\left(\log E\right) = \left\{(x,\alpha)\mid \dlog f^\alpha(x)=0\right \} \subseteq X^{\Sigma}\times \PP^{p-1}.$$
As also stated there, we have have the following containments:
$$C_F \subseteq C_{F,X^{\Sigma}} = \overline{C_{F,X^\Sigma}} \subseteq C_{F,X^{\Sigma}}(\log E).$$
A priori, $C_{F,X^{\Sigma}}(\log E)$ can have irreducible components that are contained in $\pi_1^{-1}(E)$. However, the following lemma assures that in this situation this cannot be the case.

\begin{lemma}[\cite{huh_2013}]\label{huhLemma} 
The logarithmic critical locus $C_{F,X^{\Sigma}}(\log E)$ is smooth and irreducible. Thus $C_{F,X^{\Sigma}}(\log E)$ equals $C_{F,X^\Sigma}$.
\end{lemma}
\begin{proof}
The morphism $\Omega_{ \TT^\Sigma }^1(\log(  \TT^\Sigma \setminus \TT))|_{X^\Sigma}\to\Omega_{X^\Sigma}^1(\log E)$ is surjective. This morphism is identified with the morphism $\mathcal{O}_{X^{\Sigma}}^p\to\Omega_{X^{\Sigma}}^1(\log E)$ mapping the $i$th generator to $\dlog f^{e_i}$. It follows that the kernel of this morphism has constant rank equal to $p-n$ and thus is locally free. Since $C_{F,X^{\Sigma}}(\log E)$ is the projectivization of the total space of the vector bundle corresponding to the kernel, it is smooth and irreducible.
\end{proof}
\begin{remark}
These properties of the asymptotic critical locus are not intrinsic to $X$, but depend on the chosen compactification. For example, in \cite{cohen_denham_falk_varchenko_2011}, the authors investigate the singularities of $C_{F,\PP^{n}}$ for hyperplane arrangement complements that are not closed in a tropical compactification but are closed in projective space.
\end{remark}
\begin{corollary}\label{qfa}
Let $(x,\alpha)\in C_{F,X^{\Sigma}}(\log E)\cap \pi_1^{-1}(\cap_{j\in J}E_{\tau_j})$ with $J$ maximal. Then the sum of the primitive ray generators $\{v_{\tau_j}\}_{j\in J} $ lies in $Q_{F,\alpha}$, i.e., $\sum_{j\in J}v_{\tau_j}\in Q_{F,\alpha}$.
\end{corollary}
\begin{proof}
By Lemma~\ref{huhLemma}, $C_{F,X^\Sigma}(\log E)$ is equal to $\overline{C_F}$, and $C_{F,X^\Sigma}(\log E)$ is a projective space bundle over $X^\Sigma$. Assume for the sake of notation that $J=\{1,\dots, r\}$. Take local coordinates $y_1,\dots,y_n$ on $U\subseteq X^\Sigma$ centered at $x$ in which $E_i\cap U=\{y_i=0\}$ for $i=1,\ldots,r.$ Take a trivialization of $C_{F,X^\Sigma}(\log E)$ over $U$ so that 
$$\Psi:C_{F,X^\Sigma}(\log E)\cap \pi_1^{-1}(U)\cong U\times \PP^{p-n-1}.$$
Denote the image of $(x,\alpha)$ under $\Psi$ by $(x,q)$, $q\in \PP^{p-n-1}$. Consider then the curve
$$\gamma:\Delta\to C_{F,X^\Sigma}(\log E), t\mapsto \Psi^{-1}((\underbrace{t,\dots,t}_{r\text{ times}},\underbrace{0,\dots,0}_{n-r\text{ times}},q)).$$
By construction, $\gamma(0)=(x,\alpha)$. It follows immediately from the definition of this $\gamma$ that $\ord_t(\gamma^*(\pi_1^*(f_i)))=\sum_{j\in J}\ord_{E_j}(f_i)=\sum_{j\in J}(v_{\tau_j})_i$, which proves the claim.
\end{proof}

Let $\alpha\in H_{E_{\tau}}$. The local expression for the logarithmic differential form $\dlog f^\alpha$ that we computed in Equation~\ref{dloglocal} shows that $\dlog f^\alpha$ has residue $0$ along every irreducible component of~$E_{\tau}$. Hence we can restrict $\dlog f^\alpha$ to $E_{\tau}$ to obtain
$$\dlog f^\alpha|_{E_\tau} \in H^0\left(E_\tau, \Omega_{E_\tau}^1\left(\log \left(\left(E\setminus E_{\tau}^\circ\right)\cap E_{\tau}\right)\right)\right).$$
We define
$$C_{F,X^\Sigma,E_\tau} \coloneqq \left\{(x,\alpha)\in E_\tau\times H_{E_\tau}\mid \dlog f^\alpha|_{E_\tau}(x)=0\right\}\subseteq E_\tau\times H_{E_\tau}.$$
By a local computation, one obtains the following description of $C_{F,X^\Sigma,E_\tau}$.
\begin{lemma}\label{id}
Considering $C_{F,X^\Sigma,E_\tau}$ as a subvariety of $X^\Sigma \times\PP^{p-1}$ via the inclusion\linebreak \mbox{$E_\tau\times H_{E_\tau}\hookrightarrow X^\Sigma \times\PP^{p-1}$}, $C_{F,X^\Sigma,E_\tau}$ coincides with $C_{F,X^\Sigma }(\log E)\cap \pi_1^{-1}(E_\tau)$. 
\end{lemma}
\begin{proof}
We work around a point $y\in\bigcap_{i=1}^k E_i$, with $E_\tau=E_1$, and such that $y$ is not contained in any other boundary component. We take coordinates $y_1,\dots,y_n$ on a small open $U$ around $y$ such that $E_i\cap U=\{y_i=0\}$ for $i=1,\dots,k$. The local expression for $\dlog f^\alpha$ from Equation~\ref{dloglocal} gives
$$\dlog f^\alpha=\sum_{i=1}^k \frac{y_i\theta_i(\alpha)+q_i}{y_i}dy_i+\sum_{i=k+1}^n\theta_i(\alpha)dy_i.$$
On the other hand, $y_2,\dots,y_n$ are coordinates on $E_\tau\cap U$, and the form $\dlog f^\alpha|_{E_\tau}$ has the local expression
$$\dlog f^\alpha|_{E_\tau} =\sum_{i=2}^k \frac{y_i\theta_i(\alpha)+q_i}{y_i}dy_i+\sum_{i=k+1}^n\theta_i(\alpha)dy_i.$$
This shows that $C_{F,X^\Sigma,E_\tau}$ and $C_{F,X^\Sigma }(\log E)\cap \pi_1^{-1}(E_\tau)$ are defined by the same equations.
\end{proof}
We will denote $E_\tau^\circ:= E_\tau\setminus \bigcup_{\tau'\not=\tau}E_{\tau'}=X^{\Sigma}\cap \mathcal{O}_\tau$. 
\begin{corollary}\label{maincor}
Let $\alpha\in H_{E_{\tau}}$. If $\dlog f^\alpha|_{E_\tau}$ has a zero on $E_\tau^\circ$, then $v_\tau\in Q_{F,\alpha}$.
\end{corollary}
\begin{proof}
If $\dlog f^\alpha|_{E_\tau^\circ}$ has a zero $x\in E_\tau^\circ$, then $(x,\alpha)\in C_{F,X^\Sigma,E_\tau}$, and thus by Lemma~\ref{id}, $(x,\alpha)\in C_{F,X^\Sigma}(\log E)$. The claim then follows from Corollary~\ref{qfa}.
\end{proof}
We will use the following theorem of Huh in order to apply Corollary~\ref{maincor}.
 \begin{theorem}[\cite{FK,huh_2013}]\label{thm huh} 
	Let $Z$ be a smooth very affine variety contained in a torus with coordinates $t_1,\dots,t_p$. Then there exists a Zariski dense open subset of $\CC^p$, s.t. for all $\alpha$ in this subset, the form $\sum_{i=1}^p \alpha_i\frac{dt_i}{t_i}$ has exactly $(-1)^{\dim(Z)}\chi(Z)$ many zeroes on $Z.$
\end{theorem}
In this theorem, smoothness is a necessary assumption---the failure of the statement for singular $X$ is explained in~\cite{BudurWang}. Notice that the boundary components $E_\tau^\circ$ are disjoint unions of smooth very affine varieties, since they are contained in the  torus orbit $\mathcal{O}_\tau\cong (\CC^*)^{p-1}$. Hence, Huh's theorem also applies to~$E_{\tau}^\circ$. 

Since the torus orbit $\mathcal{O}_\tau$ is naturally a quotient of $(\CC^*)^p$, its coordinate ring naturally is a subring of the coordinate ring of $(\CC^*)^p$. Under this identification, when applying Theorem~\ref{thm huh} to $E_{\tau}^\circ\subseteq \mathcal{O}_\tau$, the vector space $\CC^{p-1}$ in which the $\alpha$ live is naturally identified with $\tau^\perp$. When applying this theorem to $X$ (and $E_\tau^\circ$, resp.), we denote the Zariski dense subset that appears in the theorem by $V\subseteq \CC^p$ (and by $V_\tau\subseteq \tau^\perp$, resp.).

\begin{theorem}\label{mainthm}
Assume that $\chi(E_\tau^\circ)\not=0$. Then for all $\alpha\in V_\tau$, the primitive ray generator $v_\tau$ is contained in~$ Q_{F,\alpha}$. In particular, $H_{E_\tau}\subseteq S_F$. Moreover, these are the only codimension-one components of $S_F$.
\end{theorem}
\begin{proof}
Since $E_\tau^\circ$ is smooth and very affine, for every $\alpha\in V_\tau$ the form $\dlog f^\alpha|_{E_\tau^\circ}$ has exactly $(-1)^{\dim E_\tau}\chi(E_\tau^\circ)$ zeroes, all contained in $E_\tau^\circ$ by Theorem~\ref{thm huh}. The first claim follows from Corollary~\ref{maincor}. Since $V_\tau$ is dense inside $\tau^\perp$ and $S_F$ is closed, we conclude that $H_{E_\tau}=\PP(\tau^\perp)\subseteq S_F$. 

We prove the last claim by reasoning the other way round. 
We recall that by Proposition \ref{SFBound}, every codimension-one component of $S_F$ is of the form $\PP(\tau'^\perp)$ for some ray $\tau'\in \Sigma$.
Let $\tau'$ be some ray in $\Sigma$ and suppose that $H_{E_{\tau'}}\subseteq S_F$. Then for all $\alpha\in H_{E_{\tau'}}$, there exists $\gamma:\Delta\to C_{F,X^\Sigma}$ with $\gamma(\Delta^\circ)\in  C_F$ and $\gamma(0)\in E\times\{\alpha\}$. Choosing $\alpha$ away from $\bigcup_{\tau''\not=\tau'}H_{E_{\tau''}}$, it follows from Proposition~\ref{SFBound} that $\gamma(0)\in E_{\tau'}^\circ\times\{\alpha\}$. 

Since $\gamma(0)\in C_{F,X^\Sigma}\cap \pi_1^{-1}(E_{\tau'}^\circ)$ it follows from Lemma \ref{id} that $\gamma(0)\in C_{F,X^\Sigma,E_{\tau'}}$, and thus $\dlog f^\alpha|_{E_{\tau'}}(\pi_1(\gamma(0)))=0$. We conclude that for all $\alpha\in H_{E_{\tau'}}$ the form $\dlog f^\alpha|_{E_{\tau'}}$ has a zero on $E_{\tau'}^\circ$. Then by Theorem~\ref{thm huh}, $E_{\tau'}^\circ$ has non-zero Euler characteristic.
\end{proof}	
\begin{remark}
	Smooth very affine varieties with zero Euler characteristic can still have ``resonant" $\alpha$ for which $\dlog f^\alpha$ has a zero. In this case, strict subvarieties of some $H_{E_\tau}$ might show up as an irreducible component of~$S_F$. For relations to resonance varieties of hyperplane arrangements, we refer to \cite{cohen2012} and the references therein.
\end{remark}
 We denote by $I \triangleleft \CC[t_1^{\pm 1},\dots, t_p^{\pm 1}]$ the defining ideal of~$X$. For $\mathbf{w}\in \mathbb{R}^n$, $\init_{\mathbf{w}}(I)$ denotes the {\em initial ideal} of $I$ as defined in \cite[Section 1.6]{MS15}.
The following lemma gives a description of the components $E_\tau^\circ$ in terms of initial ideals. 
\begin{lemma}\label{lemmatorus}
	Let $\tau$ be a ray in $\Trop(X)$ with primitive ray generator $\mathbf{w}.$ Under the natural isomorphism  $\mathcal{O}_\tau \cong  (\CC^*)^p/\TT_{\mathbf{w}},$ the intersection
	$X^\Sigma \cap \mathcal{O}_\tau$ corresponds to $V(\init_{\mathbf{w}}(I))/\TT_{\mathbf{w}}.$
\end{lemma}
\begin{proof}
	For a proof, see \cite[page 308]{MS15}.
\end{proof}
Recall that we started by taking a fan $\Sigma$ supported on the tropical variety of $X$. In general, there are many such $\Sigma$ and there is no coarsest fan structure on $\Trop(X)$. Despite the lack of a coarsest fan structure, Theorem~\ref{mainthm} guarantees that some rays are present in every fan structure, namely those for which $E_\tau^\circ =X^\Sigma\cap \mathcal{O}_\tau$ has non-zero Euler characteristic. We now characterize these rays in terms of initial ideals, under a connectedness assumption.
\begin{definition}\label{def rigid}
A ray $\RR_{\geq 0}\cdot\mathbf{w}\subseteq \Trop(X)$ is called {\em rigid} if for all $\mathbf{v}\in\RR^{p}\setminus \RR\mathbf{w}$,  $\init_{\mathbf{w}}(I)\not=\init_{\mathbf{w}+\epsilon\mathbf{v}}(I)$  for all $1\gg\epsilon>0$.
\end{definition}

In order to decide {\em in practice} if a ray is rigid, one has has to compare the initial ideal of the defining ideal of $X$ w.r.t. this ray with the initial ideal w.r.t. a relative interior point of all neighboring cones.
Rigid rays can also be characterized as follows.
\begin{lemma}\label{rigidtorus}
The ray $\tau=\RR_{\geq 0}\cdot\mathbf{w}\subseteq \Trop(X)$ is rigid if and only if $E_\tau^\circ=X^\Sigma\cap \mathcal{O}_\tau $ is not invariant under any non-trivial subtorus of $\mathcal{O}_\tau$.
\end{lemma}
\begin{proof}
The ray $\tau$ is rigid if and only if $\init_{\mathbf{w}}(I)$ is not homogeneous with respect to any weight vector other than scalar multiples of $\mathbf{w}$. In other words, $\tau$ is rigid if and only if $\init_{\mathbf{w}}(I)$ is not preserved under any non-trivial subtorus of $(\CC^*)^p$ other than $\TT_\mathbf{w}:=\{(t^{w_1},\dots,t^{w_p})\mid t\in \CC^*\}$. The claim then follows from Lemma~\ref{lemmatorus}.
\end{proof}
\begin{lemma}
In the situation as above, the following two statements hold true.
\begin{enumerate}[(i)] 
\item If $\chi(E_\tau^\circ)\not=0$, then $E_\tau^\circ$ is not preserved by any non-trivial subtorus of~$\mathcal{O}_\tau$.
\item Assume that $E_\tau^\circ$ is connected. If $E_\tau^\circ$ is not preserved by any non-trivial subtorus, then $\chi(E_\tau^\circ)\not=0$.
\end{enumerate}\label{toruschi}
\end{lemma}
\begin{proof}
Denote by $\TT$ the non-trivial subtorus of $\mathcal{O}_\tau$ that preserves $E_\tau^\circ.$ The action of $\TT$ on $E_\tau^\circ$ is free and hence so is the action of the group $\mathbb{Z}/m\mathbb{Z}$ of $m$-th roots of unity for all $m\in \mathbb{Z}_{>0}.$ Hence $E_\tau^\circ \to  E_\tau^\circ / (\mathbb{Z}/m\mathbb{Z}) $ is an $m$-fold cover. Hence $\chi(E_\tau^\circ)=m\cdot \chi(E_\tau^\circ / (\mathbb{Z}/m\mathbb{Z}) ).$ In particular, $\chi(E_\tau^\circ)$ is divisible by every non-negative number $m$ and hence is $0,$ concluding the proof of the first statement.
Since $X$ is sch\"on, $E_\tau^\circ$ is smooth, and by connectedness it is thus irreducible. Let~$\mathcal{F}$ be the irreducible perverse sheaf $\iota_* (\CC_{E_\tau^\circ}[\dim E_\tau^\circ])$, where $\iota$ denotes  the inclusion $E_\tau^\circ \hookrightarrow \mathcal{O}_\tau$. By \cite[Theorem 0.3]{GL96}, $\mathcal{F}$~and hence also $E_\tau^\circ$ is preserved by a non-trivial subtorus of $\mathcal{O}_\tau$.
\end{proof}
\noindent In summary, we obtain the following implications:
\begin{align*}
\tau \text{ is rigid}  \Leftrightarrow  E_\tau^\circ \text{ is not invariant under any non-trivial subtorus of } \mathcal{O}_\tau \Leftarrow  \chi\left(E_\tau^\circ\right)\not=0,
\end{align*}
where the last implication is a bi-implication if we assume in addition that $E_\tau^\circ$ is connected. Combining this with Theorem~\ref{mainthm}, we obtain the following proposition.
\begin{proposition}\label{main}
Assume that for all $\tau\in \Sigma$ the intersection $E_\tau^\circ=X^{\Sigma}\cap \mathcal{O}_\tau$ is connected. If $\tau$ is a rigid ray, then for  all $\alpha\in V_\tau$, $v_\tau\in Q_{F,\alpha}$. In particular, \mbox{$\PP(\tau^\perp)\subseteq S_F$.} This gives a bijection of the rigid rays in $\Trop(X)$ and the codimension-one components of~$S_F$.
\end{proposition}

\begin{example}\label{ex:hpa1}
	We pick up the arrangement from Example~\ref{ex:hpa}, i.e., $X'\subseteq \CC^2$ is the complement of the arrangement defined by $f=xy(x-y)(x-1)$, the four factors of which form the tuple $F=(f_1,f_2,f_3,f_4)$. 
 Mapping $X'$ into $\CC^4$ via the tuple $F$, its image is 
$$X =V\left(t_1-t_4-1,  t_1-t_2-t_3\right).$$
By Lemma~\ref{lemmatorus}, all rays of the Gr\"obner fan on $\Trop(X)$ are rigid. They are generated by the vectors $e_2,$ $e_3,$ $e_4,$ $-e_2-e_3,$ $e_1+e_2+e_3$, and $-e_1-e_2-e_3-e_4.$  It follows from Proposition~\ref{main} that the codimension-one part of $S_F$ equals
$$\left\{s_1+s_2+s_3+s_4=0\right\} \cup  \left\{s_2+s_3=0\right\} \cup \left\{s_1+s_2+s_3=0\right\} \cup \bigcup_{i=2}^4\left\{s_i=0\right\},$$
coinciding with what was computed in Example~\ref{ex:hpa}. Notice that $\{s_1=0\}$ does not show up in $S_F$. There are two equivalent ways of explaining this: using the fact that $V(f_1)\setminus V(f_2f_3f_4)$ is isomorphic to $\CC^*$ and thus has zero Euler characteristic, or using the fact that the ray generated by $e_1$ is not rigid. We note that a compactification of $X'$ is given by $\PP^2$. To make the boundary of this compactification SNC we have to blow up the triple intersection points, namely the origin and a point on the hyperplane at infinity. Every ray $\tau$ in the Gr\"obner fan that we found above is then of the form $\PP(\tau^\perp)=H_{E_i}$ for some boundary divisor $E_i$ in this compactification.

Since $X'$ is a hyperplane arrangement complement, it makes sense to investigate the relation between the combinatorics of the arrangement and $S_F$. Recall that the combinatorics of an arrangement are encoded in the its \textit{intersection lattice} $\mathcal{L}$, which is a poset with one vertex for every subspace that can be formed by intersecting some of the hyperplanes in the arrangement. These vertices are referred to as {\em edges} of the arrangement, and  are ordered using the reverse inclusion. An edge $S$ is a \textit{dense edge} if the subposet $\mathcal{L}_{\leq S}$ is not a non-trivial product of two posets. An edge $S$ is a \textit{flacet} if neither $\mathcal{L}_{\leq S}$ nor $\mathcal{L}_{\geq S}$ is a non-trivial product of two posets.  The following diagram depicts the Hasse diagram of the arrangement:
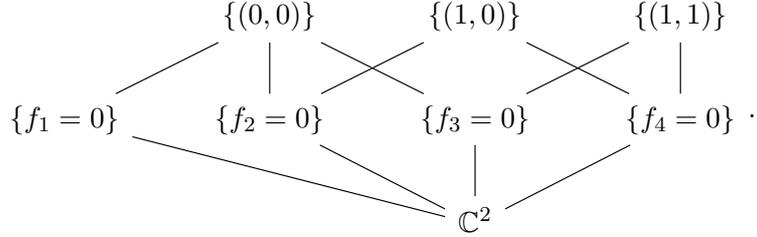
\[\begin{tikzcd}
 & \{(0,0)\}\arrow[d,dash]\arrow[dl,dash]\arrow[dr,dash] & \{(1,0)\}\arrow[dl,dash]\arrow[dr,dash] & \{(1,1)\}\arrow[d,dash]\arrow[dl,dash]\\
 \{f_1=0\}\arrow[drr,dash] & \{f_2=0\}\arrow[dr,dash] & \{f_3=0\}\arrow[d,dash] & \{f_4=0\}\arrow[dl,dash]\\
 & &\CC^2 & 
\end{tikzcd}.\]
We see from this diagram that there are $5$ dense edges: $\{x=0\},$ $\{y=0\},$ $\{x=y\},$ $\{x=1\}$, and $\{(0,0)\}$.

To understand the relation between the codimension-one components of $S_F$ and the intersection lattice, we regard $X'$ as the complement of the projective hyperplane arrangement $\tilde{f}=uv(u-v)(u-w)w$ inside~$\PP^2$ and define $f_5\coloneqq w$. Denote by $\mathcal{L}$ the intersection lattice of the affine arrangement defined by the same equation. By~\cite[Theorem 2.7]{Feichtner2005}, the rays in the Gr\"obner fan on $\Trop(X)\subseteq \RR^4$ are in bijection with the {\em flacets} of $M$. A concrete recipe to compute the rays is as follows: take a flacet~$F$ and  for $i=1,\dots,5$ define $(v_F)_i=1$ if $F\subseteq \{f_i=0\}$ and $(v_F)_i=0$ otherwise. Then add a multiple of $(1,\dots,1)$ to $v_F$ to make the last coordinate equal to~$0$, and then forget the last coordinate. In this example, all dense edges except for $\{f_1=0\}$ are flacets, and it can be verified that the procedure outlined here indeed recovers the $6$ rays that we found before.

Finally, let  $E_1\coloneqq V(f_1)$. We notice that for $\alpha\in \{s_1=0\}$, the form $\dlog f^\alpha|_{E_1}$ equals $(\alpha_2+\alpha_3)dy/y$. It follows that the form $\dlog f^\alpha|_{E_1}$ has a zero on $E_1$ if and only if $\alpha\in\{s_1=s_2+s_3=0\}$. By Corollary~\ref{maincor}, for general such $\alpha$, $e_1 \in Q_{F,\alpha}$. In conclusion, although $E_1$ does not contribute a codimension-one component to $S_F$, it does contribute an embedded component, in the sense that 
$\overline{\pi_2\left(C_{F,\PP^2}\cap \pi_1^{-1}(E_1^\circ)\right)} = \left\{s_1=s_2+s_3=0\right\}.$
\end{example}

\subsection{Maximum likelihood estimation interpretation}\label{sec:MLE}
Suppose that $X$ has non-zero Euler characteristic. Then by Theorem~\ref{thm huh}, for a general data vector $\widetilde{\alpha} $, the function $f^{
\widetilde{\alpha} }$ has exactly $(-1)^{\dim X}\chi(X)$ many critical points. This number is also called the {\em maximum likelihood degree of $X$} and is denoted by $\dml(X)$. We denote the Zariski open subset of $\PP^{p-1}$ for which this statement holds true by~$V$. Consider a rigid ray $\tau$, for which we thus know that $\chi(E_\tau^\circ)\not=0$. Since the $E_\tau^\circ$ are smooth very affine varieties as well, we denote their set of general data vectors by $V_\tau\subseteq \tau^\perp$.
By its very definition, the set $Q_{F,{\widetilde{\alpha} }}$ describes the behavior of critical points of $f^\alpha$ as $\alpha$ approaches $\widetilde{\alpha} $. Keeping this in mind, Corollary~\ref{maincor} states that as $\alpha$ approaches an ${\widetilde{\alpha} }\in V_\tau$ along a curve in $V$, at least one of the critical points of $f^\alpha$ approaches the torus orbit corresponding to $\tau$. In particular, its limit lies in $X^\Sigma \setminus X$.

The asymptotic behavior of critical points as described above is particularly explicit in the $\dml(X)=1$ case, i.e., when the signed Euler characteristic of $X$ is $1$. In this case the maximum likelihood estimate is unique and determined by the rational map
$$\psi\colon \PP^{p-1}\dashrightarrow X,$$
mapping $\alpha$ to the unique critical point of $f^{\alpha}$. Notice that $\psi$ is the rational inverse to the birational morphism $\pi_2:C_{F,Y}\to \PP^{p-1}$. Then $\psi_i \coloneqq f_i\circ \psi$ are rational functions on~$\PP^{p-1}$. We now translate Theorem~\ref{mainthm} into a statement about the $\psi_i$. Note that $v_\tau\in Q_{F,\alpha}$ if and only if there exists a morphism $\gamma:\Delta\to C_{F,X^\Sigma}$ with $\gamma(\Delta^\circ)\in  C_F$, $\gamma(0)\in (X^\Sigma\setminus X)\times\{\alpha\}$, and $\ord_t(\gamma^*(\pi_1^*F))= v_\tau$. The following diagram gives an overview of the various objects and morphisms that we have defined:
\[\begin{tikzcd}
	&& \Delta \\
	&& {C_{F,X^\Sigma}} \\
	{\mathbb{C}^p} & {X^\Sigma} && {\mathbb{P}^{p-1}} \\
	& {\mathbb{T}^\Sigma} \\
	{\mathbb{C}}
	\arrow["\gamma", from=1-3, to=2-3]
	\arrow["{\pi_2}", from=2-3, to=3-4]
	\arrow["\psi"', dashed, from=3-4, to=3-2]
	\arrow["F"', dashed, from=3-2, to=3-1]
	\arrow["{\pi_1}"', from=2-3, to=3-2]
	\arrow["\psi", dashed, from=3-4, to=4-2]
	\arrow[hook, from=3-2, to=4-2]
	\arrow["{t_i}", dashed, from=4-2, to=5-1]
	\arrow["{\mathrm{pr}_i}"', from=3-1, to=5-1]
	\arrow["{f_i}"', dashed, from=3-2, to=5-1]
	\arrow["{\psi_i}", bend left=25, dashed, from=3-4, to=5-1]
\end{tikzcd}.\]
By the definition of $\psi$, the morphisms $\pi_1\circ \gamma$ and $\psi\circ\pi_2\circ\gamma$ coincide.
It follows that
$$\ord_t\left(\left(\psi\circ \pi_2\circ \gamma\right)^*F\right) = v_\tau.$$
In other words, $\ord_t((\psi_i\circ \pi_2\circ \gamma)^*f_i)= (v_\tau)_i$. Since this is true for all $\alpha\in V_\tau$, we conclude that  $\ord_{H_{E_\tau}}(\psi_i) =(v_\tau)_i.$
Moreover, $\psi_i$ has no other zeroes or poles besides the $H_{E_{\tau}}$. To see this, we note that the same argument shows that every additional pole or zero of $\psi_i$ induces a component of $S_F$. Namely, suppose there is an additional variety $R\subseteq \PP^{p-1}$ such that $\ord_R(\psi_i)\not=0$. Then approaching a general point $p$ of $R$ along a curve $\gamma$, $\psi_i
(\gamma(t))$ will approach $0$ or $\infty$, and in particular $\psi(\gamma(t))$ leaves $X$ as $t\to 0$. Hence, $p\in S_F$ and thus $R\subseteq S_F$. However, by Theorem \ref{mainthm} all components of $S_F$ are of the form $H_{E_\tau}$ for some rigid ray $\tau\in \Trop(X)$. We deduce the following proposition.
\begin{proposition}\label{prop MLEdml1}
Let $X$ be a sch\"{o}n very affine variety with $\dml(X)=1$ and $\Sigma$ a fan whose support is $\Trop(X)$ s.t. for all rays $\tau \in \Sigma$, $X^\Sigma \cap \mathcal{O}_\tau$ is connected.
For a rigid ray $\tau \in \Sigma$, let $g_\tau$ be a defining equation of~$\tau^\perp$. Then there exist complex numbers $c_i$ such that
$$
\psi_i = c_i\prod_{\tau \text{ rigid}}g_{\tau}^{\left(v_\tau\right)_i}.
$$
Since $\psi_i$ is homogeneous, the primitive generators of the rigid rays sum up to $0$, i.e.,
$\sum_{\tau\text{ rigid}} v_\tau = 0.$
\end{proposition}
The fact that the rigid rays sum up to zero is a special property of sch\"on very affine varieties with maximum likelihood degree one, as will be demonstrated in Example~\ref{ex:conic}. Related structure results for general very affine variety with maximum likelihood degree one were obtained in~\cite{Huh14} and~\cite{DMS21}.

\begin{example}\label{ex:hpa2}
We continue Example~\ref{ex:hpa1}. In this case $\chi(X)=1$. An explicit computation of the morphism $\psi$ using {\tt Mathematica} gives the four rational functions
\begin{align*}
\psi_1 &=\frac{s_1+s_2+s_3}{s_1+s_2+s_3+s_4},\quad & \psi_2&= \frac{s_2(s_1+s_2+s_3)}{(s_2+s_3)(s_1+s_2+s_3+s_4)},\\
\psi_3&= \frac{s_3(s_1+s_2+s_3)}{(s_2+s_3)(s_1+s_2+s_3+s_4)}, \quad
&\psi_4&=\frac{-s_4}{s_1+s_2+s_3+s_4},
\end{align*}
as predicted by Proposition~\ref{prop MLEdml1}. Note that the rigid rays sum up to zero.
\end{example}
In order to formulate an analogous property for arbitrary maximum likelihood degree, we make use of the following lemma.
\begin{lemma}\label{lift}
 Let $\tau$ be a rigid ray and $\alpha\in V_\tau$.  Assume that $E_\tau^\circ=X^\Sigma\cap \mathcal{O}_\tau$ is connected and that $\dml(E_\tau^\circ)>0$. Let $S\subseteq \PP^{p-1}$ be an irreducible smooth curve containing $\alpha$ whose generic point lies in~$V$ and which intersects $H_{E_\tau}$ transversely. Then there exists a morphism $\gamma:\Delta\to C_{F,X^\Sigma}$ with $\gamma(\Delta^\circ)\in  C_F$, $\gamma(0)\in E_\tau^\circ\times\{\alpha\}$, $\pi_2(\gamma(\Delta))\subseteq S$, and $\ord_t(\gamma^*(\pi_1^*F))=v_\tau$.
\end{lemma}
\begin{proof}
Since $\alpha\in V_\tau$ and $\dml(E_\tau^\circ)>0$, $\dlog f^\alpha|_{E_\tau}$ has a non-empty, zero-dimensional, reduced zero-scheme $Z\subseteq E_\tau^\circ$, by Theorem \ref{thm huh}. Let $x\in Z$. We conclude that $(x,\alpha)\in \pi_2^{-1}(S)$. Let $C$ be an irreducible component of $C_{F,X^\Sigma}\cap \pi_2^{-1}(S)$ passing through $(x,\alpha)$. A local computation reveals that 
$$C_{F,X^\Sigma}\cap \pi_2^{-1}\left(S\right)\cap \pi_1^{-1}\left(E_\tau\right) = C_{F,X^\Sigma}\cap \pi_2^{-1}(\{\alpha\})\cap \pi_1^{-1}\left(E_\tau\right) =Z.$$
It follows that $C\cap \pi_1^{-1}(E_\tau)\subseteq Z$ is zero-dimensional, reduced, and non-empty since it contains $(x,\alpha)$. Hence any local equation $g$ for $\pi_1^{-1}(E_\tau)$ around $(x,\alpha)$ gives a generator for the maximal ideal of $\mathcal{O}_{C,(x,\alpha)}$. By Cohen's structure theorem, $t\mapsto g$ gives an isomorphism $\CC\llbracket t\rrbracket\cong \widehat{\mathcal{O}_{C,(x,\alpha)}}$ giving rise to the morphism $\gamma:\Delta\to C_{F,X^\Sigma}$ that we are looking for, since by construction $\ord_t(\gamma^* E_\tau)=1$ and thus $\ord_t(\gamma^*(\pi_1^*F))=v_\tau$.
\end{proof}
\begin{remark}
The difference between this construction and the construction in the proof of Corollary~\ref{qfa} is the fact that in this lemma we start by specifying a curve in $\PP^{p-1}$ along which we approach $\alpha$ and then lift it to $C_{F,X^\Sigma}$. Note that the number of components of $\pi_2^{-1}(S)$ passing through $(x,\alpha)$ is closely related to \cite[Conjecture 3.19]{HS14}, which essentially predicts that there is only a single component passing through~$(x,\alpha)$.
\end{remark}
We deduce the following consequence for maximum likelihood estimation. Let \mbox{$\alpha\in V_\tau$} and  $\gamma:\Delta\to \PP^{p-1}$ with $\gamma(\Delta^\circ)\in V$ approaching $\alpha$. Let  $S$ be a system of equations in $ \CC[t_1,\dots,t_p,s_1,\dots,s_p]$ defining $C_F$. Then substitute the components of $\gamma$ for the $s_i$-variables. The  system $S'$ obtained like that consists of equations in $\CC\llbracket t\rrbracket[t_1,\dots,t_p]$. A solution $\tilde{\gamma}\in \CC(\!(t)\!)^p$ of this system is a solution of the equation
$$\dlog f^{\gamma(t)}(\tilde{\gamma}(t)) =  0.$$
 Lemma~\ref{lift} assures that the system $S'$ has a solution and  that moreover it has a solution for which $\ord_t(\tilde{\gamma}_i)=(v_\tau)_i$ for all $i=1,\dots,p$. In practice, in order to approximate such a solution, one does a formal power series substitution $x_i=\sum_{j=(v_\tau)_i}^\infty c_{i,j}t^j$ into the system $S'$ and iteratively solves for the $c_{i,j}$. 
\begin{example}\label{ex:conic}
Let $g\in \CC[x,y]$ be a generic conic through $(0,0)$ and write $g=l+q$ as a sum of a linear and a quadratic part. We embed $\CC^2$ into $\CC^3$ via the tuple $F=(x,y,g)$  and denote, as usual, $f=xyg$. Denote by $X\coloneqq F(\CC^2) \cap  (\CC^*)^3$ the intersection of its image with $(\CC^*)^3$. $X$ is a hypersurface in $(\CC^*)^3$ cut out by the polynomial
$h \coloneqq  t_3 - g(t_1,t_2).$
Since $g$ is generic for its Newton polygon, $X$ is sch\"on (see \cite[Section 2]{Hov77}). The rigid rays are given by $e_1,e_2,e_3,e_1+e_2+e_3$, and $-e_1-e_2-2e_3$. 
We remark that $\dml(X)=3$ and that the rigid rays do not sum up to zero, demonstrating that the assumption on the maximum likelihood degree in Proposition~\ref{prop MLEdml1} is indeed necessary.
We now illustrate Lemma~\ref{lift} for the sch\"{o}n very affine variety $X$ as above. Consider for example the ray $\tau=-e_1-e_2-2e_3$. The initial ideal $\init_{(-1,-1,-2)}(h)$ is generated by $t_3-q(t_1,t_2)$. By Lemma~\ref{lemmatorus}, it follows that $X^\Sigma\cap \mathcal{O}_\tau$ is a copy of $\PP^1$ minus $4$ points, and thus has Euler characteristic~$-2$. By Theorem~\ref{thm huh} and Corollary~\ref{maincor}, this means that as one approaches a general $\alpha\in \{s_1+s_2+2s_3=0\}$, at least two of the three maximum likelihood estimates of $\dlog t^\alpha$ will approach the torus orbit~$\mathcal{O}_\tau$. 
To make this more explicit, consider for instance $g=x+y+x^2+xy+y^2$, which turns out to be generic enough. Take the point $(2,1,-3/2)\in \tau^\perp$ and the curve $t\mapsto (2+t,1+t,-3/2)$ approaching it. We now notice that $X$ equipped with the tuple $(t_1,t_2,t_3)$ is isomorphic, via the map $F$, to $\CC^2\setminus V(f)$ with the tuple $(x,y,g)$. To make our computations simpler we will work with the latter. On $\CC^2$ we compute that the two components of $\dlog f^{\gamma(0)}$ are zero on $\CC^2\setminus V(f)$ if and only if the following two equations are fulfilled:
\begin{align*}
x + 2tx - 2x^2 + 2tx^2 + 4y + 2ty + xy + 2txy + 4y^2 + 
2ty^2&=0,\\
2x + 2tx + 2x^2 + 2tx^2 - y + 2ty - xy + 2txy - 4y^2 + 
2ty^2&=0.
\end{align*}
The fact that $v_\tau=(-1,-1,-2)\in Q_{F,\alpha}$ implies that this system has a solution $(\eta_1,\eta_2)$ in $t^{-1}\CC\llbracket t\rrbracket$, for which $g(\eta_1,\eta_2)=c_{-2}t^{-2}+\text{ higher order terms}$.
 With {\tt Mathematica}, we compute that we indeed have a solution with
\begin{align*}
\eta_1= \frac{-7+\sqrt{33}}{(15-\sqrt{33})t}+O\left(1\right),\quad
\eta_2= \frac{-13+3\sqrt{33}}{(60-4\sqrt{33})t}+O\left(1\right).
\end{align*}
This solution converges to $(1:\frac{1}{8}(-1-\sqrt{33}):0)\in \PP^2$, which is one of the three critical points of $f^\alpha$ on $\PP^2$. The other ones are $p_2=(1:\frac{1}{8}(-1+\sqrt{33}):0)$ and $p_3=(3:-3:1)$. For $p_2$, we can construct a similar curve to the one found above. The point $p_3$ is a point in $X$ and thus we get a solution in $\CC\llbracket t\rrbracket$ whose first terms are
$$(3-74 t+3508 t^2+O\left(t^3\right),  -3+62 t-2948 t^2+O\left(t^3\right)).$$
\end{example}
We remark that in the previous example the following equality holds:
$$\sum_{\tau\text{ rigid}}\chi\left(E_\tau^\circ\right)v_\tau = 0,$$
in analogy to the second equation in~Proposition~\ref{prop MLEdml1}. It would be interesting to study if this equality holds true in general.

\section{Bernstein--Sato ideals}\label{section BernsteinSato}
In this section, we investigate Bernstein--Sato ideals and the codimension-one components of their vanishing sets. We explain how those can partially be recovered in terms of $Q_{F,\alpha}$ and formulate a conjecture relating those codimension-one components to log-canonical threshold polytopes.

\subsection{Slopes of Bernstein--Sato ideals}
Let $Y\subseteq \CC^p$ be a smooth closed subvariety the affine space. We consider the tuple of regular functions $G=(g_1,\dots,g_p)$ on $Y$ consisting of the restriction of the  coordinates $t_1,\dots,t_p$ on $\CC^p$ to $Y$. We denote their product by $g\coloneqq \prod_{i=1}^pg_i$. We denote by $X$ the very affine variety $Y\cap(\CC^*)^p$ and on it we consider the tuple of nowhere vanishing functions $F=(f_1,\dots,f_p)$ consisting of the restriction of the coordinates $t_1,\dots,t_p$ on $\CC^p$ to~$X$. For a smooth affine algebraic variety with a tuple of regular functions, the Bernstein--Sato ideal is defined as follows.
\begin{definition}
The {\em Bernstein--Sato ideal} of the tuple $G$ on $Y$, denoted by $B_G$, is the ideal consisting of all polynomials $b\in\CC[s_1,\dots,s_p]$ for which there exists a differential operator $P(s_1,\dots,s_p)\in H^0(Y,\mathcal{D}_{Y}[s_1,\dots,s_p])$ such that
$$P\bullet \left( g_1^{s_1+1}\cdots g_p^{s_p+1} \right) = b\cdot g_1^{s_1}\cdots g_p^{s_p},$$
where $\mathcal{D}_{Y}[s_1,\dots,s_p]\coloneqq \mathcal{D}_Y \otimes_{\mathcal{O}_Y} \mathcal{O}_Y[s_1,\ldots,s_p]$ is the sheaf of algebraic linear partial differential operators on~$Y$ with formal variables $s_1,\dots,s_p$ adjoined. 
\end{definition}
Sabbah~\cite{Sabbah} showed that every codimension-one component of $V(B_G)$ is an affine hyperplane. The set of {\em Bernstein--Sato slopes of $G$}, denoted by $\BS_{G}$, is defined to be the union of these hyperplanes after translating them to the origin. Since this is a homogeneous variety by definition, we will also consider the projective version of this variety, denoted by $\PP(\BS_G)\subseteq \PP^{p-1}$.
\begin{theorem}[{\cite[R\'esultat 6]{Mai16}}]\label{thm maisonobe}
Let
$$W_G \coloneqq   \left\{ \left(\sum_{i=1}^p \alpha_i\frac{dg_i}{g_i}(x),\alpha\right)\mid x\in X, \alpha\in\CC^p \right\} \subseteq T^*X\times\CC^p.$$
Then
$$\BS_G = \pi_2\left(\overline{W_G}^{T^*Y\times \CC^p}\cap V\left(\pi_1^*\left(\pi^*g\right)\right)\right),$$
where $\pi_1,\pi_2$ are the first and second projection from $T^*Y\times\CC^p$, and $\pi:T^*Y\to Y$ is the natural map.
\end{theorem}
Note that the description of $\BS_G$ in this theorem is very similar to the definition of~$S_F$. There are two differences between the objects.
In the definition of $S_F$, the second factor is equal to $\PP^{p-1}$ rather than~$\CC^p$. The second difference is the fact that in the definition of $S_F$, the set $C_F$ is closed inside a compactification of~$X$, whereas in Theorem~\ref{thm maisonobe}, the closure of $W_G$ is taken inside the non-compact affine variety~$Y$. Hence there will be contributions to $S_F$ from boundary components at infinity that are not relevant from the perspective of the Bernstein--Sato ideal. The following lemma explains how to identify contributions to $\BS_G$ in terms of~$Q_{F,\alpha}$.

\begin{lemma}\label{thm critical BS}
Let $\alpha\in \PP^{p-1}$ and denote by $L_\alpha\subseteq \CC^p$ the line through the origin corresponding to $\alpha$. If $Q_{F,\alpha}\cap \ZZ_{\geq 0}^p\not=\emptyset$, then $L_\alpha\subseteq \BS_G$.
\end{lemma}
\begin{proof}
Let $i:Y\hookrightarrow \overline{Y}$ be a compactification of $Y$, and thus also of $X$. By the assumption, there exists $\gamma:\Delta \to \overline{Y}\times\PP^{p-1}$ with $\gamma(\Delta^\circ)\in  C_F$, $\gamma(0)\in \overline{Y}\setminus X\times\{\alpha\}$, and such that moreover $\ord_t(\gamma^*(\pi_1^*f_i))\geq 0$ for $i=1,\dots,p$. Notice that $\gamma(\Delta^\circ)\in  C_F$ implies that $\pi_1(\gamma(\Delta^\circ))\in X\subset (\CC^*)^p$. Then the fact that $\ord_t(\gamma^*(\pi_1^*f_i))\geq 0$ implies that $\pi_1(\gamma(0))\in \overline{Y}\cap \CC^p=Y$. This means that $\pi_1(\gamma(0))\in Y\setminus X=V(g)$. 
Now let $\tilde{\alpha}\in L_\alpha$. By the local triviality of the tautological line bundle, we get a lift $\eta:\Delta \to \CC^p$ of $\pi_2\circ \gamma$ with $\eta(0)=\tilde{\alpha}$. Then we obtain
$$\Delta  \stackrel{(\pi_1\circ \gamma,\eta)}{\longrightarrow}  Y\times\CC^p \longrightarrow T^*_YY\times\CC^p,$$
where $T^*_YY$ denotes the zero section of $T^*Y$. By construction, for $(y,\beta)$ in the image of $(\pi_1 \circ \gamma,\eta)$, $\dlog g^\beta(y)$ is in $T_Y^*Y.$ Hence the image of the morphism  $(\pi_1 \circ \gamma,\eta)$ lies in~$\overline{W_G}$. Since $\pi_1(\gamma(0))\in V(g)$, we conclude that 
$(\pi_1(\gamma(0)),\eta(0))\in \overline{W_G}\cap V(g)$
and thus, by Theorem~\ref{thm maisonobe}, $\tilde\alpha\in \BS_{G}$.
\end{proof}
We denote by $\PP(\BS_G)$ the projectivization of the hyperplanes in $\BS_G$, and by $X^{\Sigma}$ a tropical compactification of $X$ as in Section~\ref{section MLEveryaffine}.
\begin{theorem}\label{bs-sf}
Assume that $X$ is sch\"on and that for all $\tau\in \Sigma$ the intersection $X^{\Sigma}\cap \mathcal{O}_\tau$ is connected. Then the codimension-one irreducible components of $S_F\cap \mathbb{P}(\BS_G)$ are exactly the hyperplanes $\mathbb{P}(\tau^\perp)$ for $\tau$ a rigid ray contained in~$\RR_{\geq 0}^p$.
\end{theorem}
\begin{proof}
By Proposition~\ref{main}, the codimension-one components of~$S_F$ are exactly the hyperplanes $\mathbb{P}(\tau^\perp)$ for $\tau$ a rigid ray. Again by Proposition~\ref{main}, all $\alpha \in \mathbb{P}(\tau^\perp)$  satisfy the condition of Lemma~\ref{thm critical BS}. Thus, $L_\alpha$ is contained in  $\BS_G$ and thus $\alpha$ lies in $\mathbb{P}(\BS_G)$. Hence $ \mathbb{P}(\tau^\perp)\subseteq \mathbb{P}(\BS_G).$   We now show that these indeed recover {\em all} components in the intersection.
As in Section~\ref{section MLEveryaffine}, denote by $\Sigma$  a fine enough fan structure on $\Trop(X)$ such that the closure $X^\Sigma$ of $X$ in the associated toric variety is smooth and the boundary $X^{\Sigma}\setminus X$ is a SNC divisor. Starting from an arbitrary fan we can always refine it to obtain a fan satisfying this condition, as in \cite[proof of 2.5]{HomTropVar}.

Denote by $\Sigma_{\CC^p}$ the standard fan for $\CC^p$. Let $\Sigma'$ be a refinement of $\Sigma$ of the type mentioned above for which the following is true: if the relative interiors of two cones $\sigma_1\in \Sigma'$ and $\sigma_2\in \Sigma_{\CC^p}$ intersect, then $\sigma_1\subseteq \sigma_2$. Denote by $X^{\Sigma'}$ the closure of $X$ in the associated toric variety~$\TT^{\Sigma'}$.
Let $\dot{X}$ be the variety obtained by removing from $X^{\Sigma'}$ all divisors corresponding to rays in $\Sigma'\setminus \RR_{\geq 0}^p$. This gives a log-resolution $\mu\colon \dot{X}\to Y$ of the divisor of $g$ inside~$Y$. Moreover, the irreducible components of the divisor $\mu^{-1}(V(g))$ are in bijection with the rays in $\Sigma'\cap \RR_{\geq 0}^p$. By \cite[Lemma 4.4.6]{bvwz2}, the irreducible components of $\BS_G$ are among the hyperplanes orthogonal to the rays in $\Sigma'\cap \RR_{\geq 0}^p$.
Let $C$ be any component of $S_F\cap \mathbb{P}(\BS_G)$. Since $C$ is in $\PP(\BS_G)$, we find that $C=\PP(\tau^\perp)$ for some ray $\tau\in \Sigma'\cap \RR_{\geq 0}^p$. Since $C$ is also a component of $S_F$, this ray must be rigid. It follows that $C$ is indeed of the form $\PP(\tau^\perp)$ for some rigid ray in $\RR_{\geq 0}^p$, concluding the proof.
\end{proof}
\begin{example}\label{example incomparable}
We pick up Example~\ref{ex:hpa2}. Using the library {\tt dmod.lib}~\cite{dmodlib} in the computer algebra system {\tt Singular} \cite{Singular}, we compute the Bernstein--Sato slopes to be
$$\BS_G =  \{s_1+s_2+s_3=0\}\cup\bigcup_{i=1}^4\{s_i=0\}.$$
It follows that $S_F$ intersected with the projectivized hyperplanes of $\BS_G$ equals
$$S_F\cap \PP\left(\BS_G\right) =  \left\{s_1+s_2+s_3=0\right\}  \cup \bigcup_{i=2}^4\left\{s_i=0\right\}.$$
In particular, the components $\{s_1+s_2+s_3+s_4=0\}$ and $\{s_2+s_3=0\}$ are contained in~$S_F$ but not in~$\BS_G$. This is explained by Theorem~\ref{bs-sf} since the rigid rays introducing these components to $S_F$ are $-e_1-e_2-e_3-e_4$ and $-e_2-e_3$, which are not contained in~$\RR_{\geq 0}^4$. 
On the other hand, the component $\{s_1=0\}$ is contained in $\BS_G$ but not in $S_F$. This shows that in general the sets of components of $\BS_G$ and $S_F$ are incomparable.  
Finally, we combine Theorem~\ref{bs-sf}  with Proposition \ref{prop MLEdml1}. We then find that the linear forms appearing in the numerators and denominators of the MLE as determined in Example~\ref{ex:hpa2} are among the defining equations for the hyperplanes in~$\BS_G$. This explains the geometry behind the observation made in~\cite{SatStu19} suggesting a link between the MLE and the Bernstein--Sato ideal of the parametrization of the algebraic model of the statistical experiment.
\end{example}

\subsection{LCT-polytopes}\label{subsection LCT}
In the previous subsection, we partially recovered the Bernstein--Sato slopes in terms of the critical slopes of the very affine variety. We now formulate a conjecture related to the translations using {\em log-canonical threshold} (LCT) polytopes.
\begin{definition}
Let $\mu:Y'\to Y$ be a log-resolution of $V(g)$ with $\mu^{-1}(V(g))=\bigcup_{i=1}^{q} E_i$. Denote by $a_{ij}\coloneqq \ord_{E_i}(g_j)$ and by $k_i\coloneqq \ord_{E_i}(K_{Y'/Y})+1$, where $K_{Y'/Y}$ denotes the relative canonical divisor.
The {\em LCT-polytope} of $G$ is 
$$\LCT(G)  \coloneqq  \left\{(s_1,\dots,s_p)\in \RR_{\geq 0}^p\mid \sum_{j=1}^p a_{ij}s_j\leq k_i \text{ for }i=1,\dots,q\right\} \subseteq  \RR_{\geq 0}^p.$$
\end{definition}
We call an affine hyperplane $H\subseteq \CC^p$ \textit{facet-defining} if $\dim (H\cap \LCT(G))=p-1$ and $H\cap \LCT(G)\subseteq \partial(\LCT (G))$. The facet-defining hyperplanes are thus all of the form
$$\left\{\ord_{E_i}\left(g_1\right)s_1+\dots + \ord_{E_i}\left(g_p\right)s_p=k_i\right\},$$
but in general not all these hyperplanes are facet-defining. The relevance of the facet-defining hyperplanes is highlighted in the following theorem:
\begin{theorem}[\cite{cassou-nogues_libgober_2010},{\cite[Corollary 1.5]{BVW21}}]\label{LCTBS}
If $\{\ord_{E_i}(f_1)s_1+\dots + \ord_{E_i}(f_p)s_p=k_i\}$ is a facet-defining hyperplane and $\ord_{E_i}(f_j)\not=0$ for all $j=1,\dots,p,$ then
$\{\ord_{E_i}(f_1)s_1+\dots + \ord_{E_i}(f_p)s_p=-k_i\}$ is an irreducible component of $V(B_G)$.
\end{theorem}
It follows from the proof of~Theorem~\ref{bs-sf} that if $X$ is sch\"on we can construct a log resolution $\mu:\dot{X}\to Y$ of the divisor of $g$ as the closure of $X$ inside a toric variety with fan $\Sigma\subseteq \mathbb{R}_{\geq 0}^p\cap \Trop(X)$. In this log resolution the components of $\mu^{-1}(V(g))$ correspond to the rays in $\Sigma$, and if  $\tau$ is a ray in $\Sigma$, $\ord_{E_\tau}(g_j)=(v_\tau)_j$. We denote by $k_\tau$ the integer $\ord_{E_\tau}(K_{X^\Sigma/X})+1$. It follows that
$$\LCT(G) =  \left\{(s_1,\dots,s_p)\in \RR_{\geq 0}^p\mid \sum_{j=1}^p \left(v_{\tau_i}\right)_js_j\leq k_{\tau_i} \text{ for }i=1,\dots,q\right\}.$$
By Theorem~\ref{bs-sf},  for every rigid ray $\tau$ in $\RR_{\geq 0}^p$, $\tau^\perp\in \BS_G$, i.e., $\tau^\perp$ is a Bernstein--Sato slope of $G$. This means precisely that at least one affine translate of $\tau^\perp$ lies in $V(B_G)$. A natural candidate for such a translate is the one corresponding to the log-canonical threshold, leading to the following conjecture.
\begin{conjecture}
For every rigid ray $\tau\in \Trop(X)\cap \RR_{\geq 0}^p$, $\{(v_\tau)_1s_1+\dots+(v_\tau)_ps_p=k_\tau\}$  is facet-defining.
\end{conjecture}
\begin{example}
Let $Y\subseteq \CC^p$ be a linear subspace such that the restrictions of the coordinate functions to $Y$ define a central indecomposable hyperplane arrangement. In this situation, the rigid rays in $\Trop(X)$ are $(1,\dots,1)$ and $(-1,\dots,-1)$. These rays are rigid since $X$ is homogeneous and they are the only rigid rays since the arrangement is indecomposable. The matroid polytope of the matroid associated to the arrangement is the intersection of $\LCT(G)$ with the hyperplane $H\coloneqq \{s_1+\dots+s_p=\dim Y\}$. It is shown in \cite{Feichtner2005} that the matroid polytope has dimension $p-1$. Hence $H$ is facet-defining and by Theorem \ref{LCTBS} $H$ is an irreducible component of $V(B_G)$. This proves \cite[Conjecture 3]{BudurBSLS} for complete factorizations of hyperplane arrangements. Recall that a factorization of a hyperplane arrangement is {\em complete} if each of its factors is linear.
\end{example}

\section{Example: Flipping a biased coin}\label{section examples}
In this section, we pick up and continue the example of flipping a biased coin twice that was studied in \cite{DMS21} and~\cite{SatStu19}. 
Consider the smooth curve $\overline{X}$ in~$\PP ^2$ defined by the homogeneous polynomial  
$$f= \det \begin{pmatrix}
p_0 & p_1\\
p_0+p_1 & p_2
\end{pmatrix}=p_0p_2-(p_0+p_1)p_1.$$ 
As pointed out in~\cite[Example 2]{DMS21}, this is the implicit representation
of the statistical model describing the following experiment: {\em Flip a biased coin. If it shows head, flip it again}. Here, the $p_0,p_1$ and $p_2$ are to be thought of as representing the probabilities of three possible outcomes of the experiment. Since these outcomes must sum up to one, we impose the additional condition that $p_0+p_1+p_2\not=0$ and hence consider the variety $\overline{X}\setminus \mathcal{H}$, where $\mathcal{H}$ is the collection of hyperplanes $\{p_0p_1p_2(p_0+p_1+p_2)=0 \}$ in~$\PP^2$. We embed $ \overline{X}\setminus \mathcal{H}$ into the 3-torus via the morphism
$$ P\colon \overline{X}\setminus \mathcal{H} \to \left(\CC^*\right)^3, \quad  \left(p_0:p_1:p_2\right)\mapsto  \left( \frac{p_0}{p_0+p_1+p_2}, \frac{p_1}{p_0+p_1+p_2}, \frac{p_2}{p_0+p_1+p_2}\right).$$
 We denote by $X$ the image of $P$ and by $Y$ the closure of $X$ inside $\CC^3$. The ideal defining $X$ is generated by the two polynomials
 $t_0t_2-(t_0+t_1)t_1$ and  $t_0+t_1+t_2-1.$
This ideal of $\CC[t_0,t_1,t_2]$ is prime  and hence $Y$ is defined by the same equations. Using {\tt Gfan}~{\cite{Gfan}, we compute that the primitive ray generators in the Gr\"{o}bner fan of~$Y$ are given by the three rows $w_1,w_2,w_3$ of the matrix
\begin{align*}
\begin{pmatrix}
2 & 1 &0\\
0 & 1 & 1\\
-2 & -2 & -1\\ 
\end{pmatrix}  \eqqcolon   \begin{pmatrix}
w_1\\
w_2\\
w_3
\end{pmatrix} .
\end{align*}
Since the tropical variety of $Y$ is one-dimensional, all three rays are indeed rigid. Hence, by Proposition~\ref{main}, the hyperplanes orthogonal to them recover the codimension-one components of $S_F$, where $F=(t_0|_X,t_1|_X,t_2|_X)$. In other words, the codimension-one part of $S_F$ is equal to
$$\{2s_0+s_1=0\}\cup \{s_1+s_2=0\}\cup\{2s_0+2s_1+s_2=0\}.$$
This model has maximum likelihood degree one. We denote the maximum likelihood estimator by $\psi \colon \PP^2\dashrightarrow X.$
As described in Proposition~\ref{prop MLEdml1},  each $\psi_i= f_i \circ \psi$ is a rational function on~$\PP^2$ whose numerator and denominator are products of the linear terms $2s_0+s_1$, $s_1+s_2$, and $2s_0+2s_1+s_2$. 
More precisely, there exist complex constants $c_1,c_2,c_3$ such that
$$\psi\left(s_0,s_1,s_2\right) = \left(c_1 \frac{\left(2s_0+s_1\right)^{2}}{\left(2s_0+2s_1+s_2\right)^{2}}, \, c_2\frac{\left(2s_0+s_1\right) \left(s_1+s_2\right)}{\left(2s_0+2s_1+s_2\right)^2}, \, c_3\frac{\left(s_1+s_2\right)}{\left(2s_0+2s_1+s_2\right)} \right).$$
For $c_1 = c_2=c_3=1$, this recovers the MLE computed in~\cite{DMS21}.
Finally, we compute the Bernstein--Sato ideal of the tuple of coordinate functions on~$Y$. Under the isomorphism
$$\CC\stackrel{\cong}{\longrightarrow} Y,\quad x\mapsto \left(x^2,x(1-x),1-x\right),$$
the tuple $G=(t_0|_Y,t_1|_Y,t_2|_Y)$ on $Y$ corresponds  to the tuple $(x^2,x(1-x),1-x)$ on~$\CC$. Using {\tt Singular}, we find that the Bernstein--Sato ideal of this tuple is generated by
$$ \prod_{k=1}^3\left(2s_0+s_1+k\right)\cdot \prod_{l=1}^2\left(s_1+s_2+l\right)$$
and thus the Bernstein--Sato slopes of $G$ are
$\BS_G = V\left(2s_0+s_1\right)\cup V\left(s_1+s_2\right).$
Indeed, the components correspond to $w_1$ and~$w_2$. The ray~$w_3$ does not contribute to the Bernstein--Sato slopes since it is not contained in~$\RR_{\geq 0}^3$.

\bigskip
\subsection*{Funding} The second author is supported by a PhD Fellowship from FWO (Research Foundation - Flanders).
\subsection*{Acknowledgments}
We are grateful to Nero Budur, Johannes Nicaise, Yue Ren, and Bernd Sturmfels for insightful discussions. We thank the anonymous referees for carefully reading our article and the valuable feedback.

\smallskip
\bibliography{literature}

\begin{thebibliography}{10}

\bibitem{Bath19}
D.~Bath.
\newblock Combinatorially determined zeroes of {B}ernstein--{S}ato ideals for
  tame and free arrangements.
\newblock {\em J. Singul.}, 20:165--204, 2020.

\bibitem{BMM}
J.~Brian\c{c}on, P.~Maisonobe, and M.~Merle.
\newblock \'{E}ventails associ\'{e}s \`{a} des fonctions analytiques.
\newblock {\em Tr. Mat. Inst. Steklova}, 238:70--80, 2002.

\bibitem{BudurBSLS}
N.~Budur.
\newblock Bernstein--{S}ato ideals and local systems.
\newblock {\em Ann. Inst. Fourier (Grenoble)}, 65(2):549--603, 2015.

\bibitem{BVW21}
N.~Budur, R.~van~der Veer, and A.~Van~Werde.
\newblock Estimates for zero loci of {B}ernstein--{S}ato ideals.
\newblock Preprint arXiv:2111.03334, 2021.

\bibitem{bvwz}
N.~{Budur}, R.~{van der Veer}, L.~{Wu}, and P.~{Zhou}.
\newblock Zero loci of {B}ernstein--{S}ato ideals.
\newblock {\em Invent. Math.}, 225:45--72, 2021.

\bibitem{bvwz2}
N.~{Budur}, R.~{van der Veer}, L.~{Wu}, and P.~{Zhou}.
\newblock Zero loci of {B}ernstein--{S}ato ideals-{II}.
\newblock {\em Selecta Math. (N.S.)}, 27(32), 2021.

\bibitem{BudurWang}
N.~Budur and B.~Wang.
\newblock Bounding the maximum likelihood degree.
\newblock {\em Math. Res. Lett.}, 22(6), 2014.

\bibitem{cassou-nogues_libgober_2010}
P.~Cassou-Nogu\'{e}s and A.~Libgober.
\newblock Multivariable {H}odge theoretical invariants of germs of plane
  curves.
\newblock {\em J. Knot Theory Ramifications}, 20(06):787--805, 2011.

\bibitem{cohen_denham_falk_varchenko_2011}
D.~Cohen, G.~Denham, M.~Falk, and A.~Varchenko.
\newblock Critical points and resonance of hyperplane arrangements.
\newblock {\em Canad. J. Math.}, 63(5):1038–1057, 2011.

\bibitem{cohen2012}
D.~C. Cohen, G.~Denham, M.~Falk, and A.~Varchenko.
\newblock Vanishing products of one-forms and critical points of master
  functions.
\newblock In {\em Arrangements of Hyperplanes — Sapporo 2009}, pages 75--107,
  Tokyo, Japan, 2012. Mathematical Society of Japan.

\bibitem{Singular}
W.~Decker, G.-M. Greuel, G.~Pfister, and H.~Sch\"onemann.
\newblock {\sc Singular} {4-1-3} --- {A} computer algebra system for polynomial
  computations.
\newblock \url{http://www.singular.uni-kl.de}, 2020.

\bibitem{DMS21}
E.~Duarte, O.~Marigliano, and B.~Sturmfels.
\newblock Discrete {S}tatistical {M}odels with {R}ational {M}aximum
  {L}ikelihood {E}stimates.
\newblock {\em Bernoulli}, 27(1):135--154, 02 2021.

\bibitem{Feichtner2005}
E.~M. Feichtner and B.~Sturmfels.
\newblock Matroid polytopes, nested sets and {B}ergman fans.
\newblock {\em Port. Math. Nova S\'{e}rie}, 62(4):437--468, 2005.

\bibitem{FK}
J.~Franecki and M.~Kapranov.
\newblock The {G}auss map and a noncompact {R}iemann--{R}och formula for
  constructible sheaves on semiabelian varietes.
\newblock {\em Duke Math. J.}, 104(1):171--180, 2000.

\bibitem{GL96}
O.~Gabber and F.~Loeser.
\newblock Faisceaux pervers $\ell$-adiques sur un tore.
\newblock {\em Duke Math. J.}, 83(3):501--606, 1996.

\bibitem{HomTropVar}
P.~Hacking.
\newblock The homology of tropical varieties.
\newblock {\em Collect. Math.}, 59:263--273, 2007.

\bibitem{Hov77}
A.~Hovanski\v{i}.
\newblock Newton polyhedra and toroidal varieties.
\newblock {\em Funkcional'nyi Analiz i ego Prilo\v{z}enija}, 11:56--64, 1977.

\bibitem{huh_2013}
J.~Huh.
\newblock The maximum likelihood degree of a very affine variety.
\newblock {\em Compos. Math.}, 149(8):1245–1266, 2013.

\bibitem{Huh14}
J.~Huh.
\newblock Varieties with maximum likelihood degree one.
\newblock {\em J. Algebr. Stat.}, 5:1--17, 2014.

\bibitem{HS14}
J.~Huh and B.~Sturmfels.
\newblock {Likelihood geometry}.
\newblock In {\em Combinatorial algebraic geometry}, volume 2108 of {\em
  Lecture notes in mathematics}, pages 63--117. Springer, New York, 2014.

\bibitem{Gfan}
A.~N. Jensen.
\newblock Gfan, a software system for {G}r\"obner fans and tropical varieties.
\newblock Available at
  \url{http://home.imf.au.dk/jensen/software/gfan/gfan.html}.

\bibitem{dmodlib}
V.~Levandovskyy and J.~Mart\'{i}n-Morales.
\newblock dmod\_lib: A {{\tt Singular:Plural}} library for algorithms for
  algebraic {$D$}-modules.
\newblock \url{https://www.singular.uni-kl.de/Manual/4-2-0/sing\_537}.

\bibitem{OnTropComp}
M.~Luxton and Z.~Qu.
\newblock Some results on tropical compactifications.
\newblock {\em Trans. Amer. Math. Soc.}, 363(9):4853--4876, 2011.

\bibitem{MS15}
D.~Maclagan and B.~Sturmfels.
\newblock {\em {Introduction to tropical geometry}}, volume 161 of {\em
  Graduate studies in mathematics}.
\newblock American Mathematical Society, Providence, R.I., 2015.

\bibitem{Mai16}
P.~Maisonobe.
\newblock Filtration relative, l'id\'eal de {B}ernstein et ses pentes.
\newblock 2016.
\newblock hal-01285562v2.

\bibitem{maisHPA}
P.~{Maisonobe}.
\newblock L'id{\'e}al de {B}ernstein d'un arrangement libre d'hyperplans
  lin{\'e}aires.
\newblock arXiv:1610.03356, 2016.

\bibitem{RT18}
J.~I. Rodriguez and X.~Tang.
\newblock Probabilistic algorithm for computing data-discriminants of
  likelihood equations.
\newblock {\em Proceedings of the 2015 ACM on International Symposium on
  Symbolic and Algebraic Computation}, 2015.

\bibitem{Sabbah}
C.~Sabbah.
\newblock Proximit\'e \'evanescente. {I}. {L}a structure polaire d'un
  {$\mathcal {D}$}-module.
\newblock {\em Compos. Math.}, 62(3):283--328, 1987.

\bibitem{SatStu19}
A.-L. Sattelberger and B.~Sturmfels.
\newblock {$D$}-modules and holonomic functions.
\newblock arXiv:1910.01395, 2019.

\bibitem{ST20}
B.~Sturmfels and S.~Telen.
\newblock Likelihood equations and scattering amplitudes.
\newblock {\em Algebraic Statistics}, 12(2):167--186, 2021.

\bibitem{Tev}
J.~Tevelev.
\newblock Compactifications of subvarieties of tori.
\newblock {\em Amer. J. Math.}, 129(4):1087--1104, 2007.

\bibitem{Wu20}
L.~Wu.
\newblock Bernstein--{S}ato ideals and hyperplane arrangements.
\newblock {\em J. Pure Appl. Algebra}, 226:106987, July 2022.

\end{thebibliography}
\bibliographystyle{abbrv}
\bigskip

\enddocument